\documentclass[11pt]{article}

\usepackage[margin=1in]{geometry}
\usepackage{amsmath,amsthm,amssymb,amsfonts,float}
\usepackage{subcaption}

\makeatletter
\newcommand{\subjclass}[2][2010]{%
  \let\@oldtitle\@title%
  \gdef\@title{\@oldtitle\footnotetext{#1 \emph{Mathematics subject classification.} #2}}%
}
\makeatother

\usepackage[colorlinks=true, pdfstartview=FitV, linkcolor=blue,citecolor=blue, urlcolor=blue]{hyperref}

\usepackage[abbrev,lite,nobysame]{amsrefs}
\usepackage{times}
\usepackage[usenames,dvipsnames]{color}

\usepackage{esint}

\usepackage{mathtools,enumitem,mathrsfs}

\usepackage[compact]{titlesec}

\usepackage[title]{appendix}

\usepackage{comment}

\newcommand{\eps}{\epsilon}

\newcommand{\abs}[1]{\left| #1 \right|}
\newcommand{\set}[1]{\left\{ #1 \right\}}

\renewcommand{\S}{\mathbb{S}}

\newcommand{\dee}{\mathrm{d}}

\usepackage{bbm}

\newcommand{\EE}{\mathbf E}
\newcommand{\PP}{\mathbf P}

\newtheorem{theorem}{Theorem}[section]

\newtheorem{corollary}[theorem]{Corollary}
\newtheorem{lemma}[theorem]{Lemma}
\newtheorem*{lemma*}{Lemma}

\theoremstyle{definition}

\newtheorem{remark}[theorem]{Remark}
\newtheorem{example}[theorem]{Example}

\numberwithin{equation}{section}
    
\begin{document}

\title{A quantitative dichotomy for Lyapunov exponents of non-dissipative SDEs with an application to electrodynamics} 
\author{Jacob Bedrossian\thanks{\footnotesize Department of Mathematics, University of California, Los Angeles, CA 90095, USA \href{mailto:jacob@math.ucla.edu}{\texttt{jacob@math.ucla.edu}}. Both authors were supported by NSF Award DMS-2108633.} \and Chi-Hao Wu\thanks{\footnotesize Department of Mathematics, University of California, Los Angeles, CA 90095, USA \href{mailto:chwu93@math.ucla.edu}{\texttt{chwu93@math.ucla.edu}}.}}

\maketitle

\begin{abstract}
In this paper we derive a quantitative dichotomy for the top Lyapunov exponent of a class of non-dissipative SDEs on a compact manifold in the small noise limit. 
Specifically, we prove that in this class, either the Lyapunov exponent is zero for all noise strengths, or it is positive for all noise strengths and that the decay of the exponent in the small-noise limit cannot be faster than linear in the noise parameter.
As an application, we study the top Lyapunov exponent for the motion of a charged particle in randomly-fluctuating magnetic fields, which also involves an interesting geometric control problem. 
\end{abstract}

\setcounter{tocdepth}{1}
{\small\tableofcontents}

\section{Introduction}\label{sec:Introduction}

The aim of this paper is twofold. First, we give a quantitative dichotomy for the (top) Lyapunov exponent of non-dissipative SDEs on a compact manifold with noise parameter $\eps > 0$. 
For this family of SDEs, the Lyapunov exponent is either zero for all $\eps > 0$, or strictly positive for all $\eps > 0$ and is bounded from below as $\lambda \gtrsim \eps$ as $\eps \to 0$.
Second, as an application of our result, we study conditions under which the motion of a charged particle in stochastically fluctuating magnetic fields has a positive top Lyapunov exponent.

Throughout, we will assume $(M,g)$ is an n-dimensional compact, connected, smooth Riemannian manifold, unless otherwise stated. For the metric on the tangent bundle $TM$, we choose the Sasaki metric induced by $g$.
Let $x_t\in M$ be an It\^o diffusion driven by smooth vector fields $X_0, X_1,\dotsc, X_r$ with a small noise parameter $\eps \geq 0$, 
\begin{equation}\label{bprocess}
    dx_t = X_0(x_t)\,dt + \sqrt{\eps} \sum_{k=1}^r X_k(x_t)\circ dW_t^k,
\end{equation}
where $\circ$ denotes the Stratonovich formulation and $\{W^k_t\}_{k=1}^r$ are i.i.d. one-dimensional Brownian motions defined on the canonical stochastic basis $(\Omega,\mathbf{P}, \mathcal{F},\mathcal{F}_t)$.
 It is a standard result that the smoothness of the vector field implies that the diffusion process generates a flow of diffeomorphisms almost surely \cite{K97}.
For each $\omega\in\Omega$, denote $\phi_{t,\omega}: M\rightarrow M$ as the flow of diffeomorphisms associated with the noise path $\omega$; i.e.
\begin{equation}
\phi_{t,\omega}(x) = x_t,
\end{equation}
where $x_t$ solves \eqref{pprocess}. 
Denote $D_x\phi_{t,\omega}$ as its linearization at $x\in M$.
If the Markov process $(x_t)$ has a unique stationary measure $\mu$ and a mild integrability condition is satisfied, then the Furstenberg-Kesten theorem \cite{FK60}
guarantees a unique top Lyapunov exponent $\lambda_1$ and a total Lyapunov exponent $\lambda_{\Sigma}$,
  such that for $\mu \times \PP$-a.e. $x$ and $\omega$, the following limits hold  
\begin{align*}
\lambda_1(\eps)  & = \lim_{t\rightarrow\infty} \frac{1}{t} \log \abs{D_x \phi_{t,\omega}} \\
\lambda_\Sigma(\eps)  & = \lim_{t\rightarrow\infty} \frac{1}{t} \log \abs{\rm{det} D_x \phi_{t,\omega}}. 
\end{align*} 
See Section \ref{sec:LEprelims} for more details. 

Denote $\mathbb{S}M$ as the unit tangent bundle, where each fiber is the unit sphere in the tangent space,
\begin{equation*}
\mathbb{S}_x M = \left\{ z \in T_xM: g_x(z, z) = 1 \right\}.
\end{equation*} 
For $z \in \mathbb{S}_{x}M$, denote
\begin{align*}
z_t = \frac{D_x \phi_{t,\omega}(x) z}{\abs{D_x \phi_{t,\omega}(x) z}}. 
\end{align*}
We define the \emph{projective process} on the sphere bundle $w_t = (x_t,z_t) \in \mathbb{S} M$, which is a standard tool in dynamical systems for studying Lyapunov exponents. 
The process $(w_t)$ is an It\^o diffusion solving the following SDE: 
\begin{equation}\label{pprocess}
    dw_t = \tilde{X}_0(w_t)\,dt + \sqrt{\eps}\sum_{k=1}^r\tilde{X}_k(w_t)\circ dW_t^k,
\end{equation}
where $\tilde{X}_k$'s are defined as
\begin{equation}\label{pfield}
    \tilde{X}_k(w) = \begin{pmatrix}
    X_k(x) \\  \left(1 - z \otimes z\right)\,\nabla_{x}X_k(x) z
    \end{pmatrix}.  
\end{equation}
Throughout, we will assume that the projective process satisfies the parabolic H\"{o}rmander condition and is irreducible (see Section \ref{sec:LEprelims} for more details). 
\begin{enumerate}
\item[(H)] The collection $\{\widetilde{X}_0; \widetilde{X}_1, \dotsc, \widetilde{X}_r\}$ satisfies the parabolic H\"ormander condition in $\S M$; i.e. the Lie algebra generated by $\{[\tilde{X}_0, \tilde{X}_k], \tilde{X}_k\}_{k=1}^r$  spans the tangent space of $\S M$ at every point. 
\item[(I)] The process $w_t^{\epsilon}$ is irreducible in the sense that $\exists t>0$ such that  $\forall w\in \mathbb{S}M$ and all open sets $U\subseteq \mathbb{S}M$,
\begin{equation*}
P_t(w, U) >0.
\end{equation*}
\end{enumerate}
Note that hypoellipticity and irreducibility of $(w_t)$ imply the same for $(x_t)$. 
A consequence of (H) and (I) is that the stationary measures are unique and admit smooth, strictly positive densities with respect to the Lebesgue measure; we denote $\rho^{\epsilon}$ and $f^{\epsilon}$ for the stationary densites of $x_t^{\epsilon}$ and $w_t^{\epsilon}$ respectively.
Note that some authors take $\epsilon$ instead of $\sqrt{\epsilon}$ for the scaling of the strength of the noise; we will make corresponding adjustment when stating the related quantitative results. 

The small noise asymptotics of Lyapunov exponents have been studied in several previous works. For instance, \cite{APW86} studied a random oscillator and found $\lambda_1^{\epsilon} = O(\epsilon)$, \cite{BG02} studied the perturbation of a two-dimensional Hamiltonian system and found $\lambda_1^{\epsilon} = O(\epsilon^{1/3})$, and \cite{DNR11} studied the perturbation of a non-Hamiltonian system and found $\lambda_1^{\epsilon} = O(\epsilon)$. For more work in this direction, see \cite{AM82}, \cite{PW07}, \cite{BG01} and \cite{S00}. 
The work \cite{BBPS22} introduced a new method for estimating the Lyapunov exponents of SDEs based on regularity estimates of the projective density $f^\eps$ and used the method to prove the lower bound $\liminf_{\epsilon\rightarrow 0} \eps^{-1} \lambda = \infty$ for a range of SDEs, including Lorenz 96 and Galerkin truncation of the 2D Navier-Stokes equations \cite{BPS22}. 

To our knowledge, there is no example in the literature of an SDE with a Lyapunov exponent that behaves as $0 < \lambda \ll \eps$ as $\eps \to 0$. 
Our main theorem addresses whether or not this is possible for non-dissipative SDEs; specifically, our result excludes the possibility of superlinear decay as $\eps \to 0$ for the class of SDE \eqref{bprocess}.  

\begin{theorem}[Quantitative dichotomy]\label{maintheorem}
Let $(M,g)$, $(x_t)$, and $(w_t)$ be as above and additionally assume that the SDE is non-dissipative, i.e.
\begin{align*}
\lambda_\Sigma(\eps) = 0. 
\end{align*}
Assume $(H)$ and $(I)$ holds. 
For each $\epsilon>0$, denote the associated top Lyapunov exponent $\lambda_1^{\epsilon}$. Then exactly one of the following holds: 
\begin{enumerate}
    \item $\lambda_1^{\epsilon} > 0$ for all $\epsilon>0$ and there exists a constant $C>0$ such that $\liminf_{\epsilon\rightarrow 0} \frac{\lambda_1^{\epsilon}}{\epsilon} \geq C$;
    \item $\lambda_1^{\epsilon}=0$ for all $\epsilon>0$.
\end{enumerate}
\end{theorem}

Another natural question is whether it is possible to introduce additional independent noise to a system with a positive Lyapunov exponent and reduce the Lyapunov exponent to zero. The following corollary is a result of the proof of Theorem \ref{maintheorem} (note that no smallness assumption is made).  
\begin{corollary} \label{cor:NewNoise}
Let $M$ be a connected, compact, Riemannian manifold as above. Suppose the diffusion process
\begin{equation*}
dx_t = X_0(x_t)\,dt + \sum_{k=1}^r X_k(x_t)\circ dW_t^k
\end{equation*}
satisfies the non-degeneracy conditions (H) and (I), and 
\begin{equation*}
dy_t =  X_0(y_t)\,dt + \sum_{k=1}^r X_k(y_t)\circ dW_t^k + X_{r+1}(y_t)\circ W_t^{r+1},
\end{equation*}
where $\{W_t^k\}_{k=1}^{r+1}$ are independent Brownian motions. Denote $\lambda_1(x)$, $\lambda_1(y)$, $\lambda_{\Sigma}(x)$ and $\lambda_{\Sigma}(y)$, the top Lyapunov exponents and the summations of the Lyapunov exponents for $x_t$ and $y_t$, respectively. Then under the assumption that $\lambda_{\Sigma}(x) = \lambda_{\Sigma}(y) = 0$, we have
\begin{equation*}
\lambda_1(x) > 0 \quad\Rightarrow\quad \lambda_1(y) >0.
\end{equation*}
\end{corollary}

As an application, we study the motion of a charged particle in a fluctuating magnetic field assumed to be parallel to the $z$-axis.    
The motion of a charged particle in a magnetic field is governed by the Lorentz force, 
\begin{align*}
\dot{\mathbf{x}} &= \mathbf{v} \\
\dot{\mathbf{v}} &= \mathbf{v}\times\mathbf{B(\mathbf{x}, t)}, 
\end{align*}
where $\mathbf{x}=(x,y,z)$ and $\mathbf{v}=(v_x,v_y,v_z)$ are vectors in $\mathbb{R}^3$ for the position and the velocity, and $\mathbf{B}$ is the magnetic field. Since the magnetic field is parallel to the $z$ axis, we consider only the planar motion in the $x$-$y$ plane (as the motion of the particle in the $z$ direction is trivial). 
In particular, we assume 
\begin{equation*}
\mathbf{B} = \begin{bmatrix}
0 \\
0 \\
B(x,y,t)
\end{bmatrix},
\end{equation*}
and consider the dynamics on the $x$-$y$ plane; the system is then reduced to
\begin{align*}
\dot{x} &= v_x \\
\dot{y} &= v_y \\
\dot{v}_x &= v_yB \\
\dot{v}_y &= -v_xB.
\end{align*}
In fact, we can reduce the system further by observing that the magnetic field does not change the kinetic energy of the particle. 
Hence, we an express the velocity in polar coordinates, $(v_x, v_y) = (r\cos{\theta}, r\sin{\theta})$ and without loss of generality assume $r=1$ (unit speed) to obtain the following system
\begin{align*}
\dot{x} &= \cos{\theta} \\
\dot{y} &= \sin{\theta} \\
\dot{\theta} &= B(x,y, t). 
\end{align*}
Using stochasticity to model a fluctuating time-inhomogeneity, our main application is the following.  

\begin{theorem}\label{application}
Consider the diffusion process
\begin{align*}
dx_t &= \cos{\theta_t}\,dt \\[3pt]
dy_t &= \sin{\theta_t}\,dt \\[3pt]
d\theta(t) &= dt + \sum_{k=1}^5 
\sqrt{\epsilon}\,B_k(x_t,y_t)\circ dW_t^k,
\end{align*}
where $B_1 = 1$, $B_2 = \sin{x}$, $B_3 = \cos{x}$, $B_4 = \sin{y}$ and $B_5 = \cos{y}$. Then 
\begin{equation*}
\lambda_1^{\epsilon} > 0 \quad\forall \epsilon>0 \quad\mbox{and}\quad \liminf_{\epsilon\rightarrow 0} \frac{\lambda_1^{\epsilon}}{\epsilon} \geq C > 0.
\end{equation*}
\end{theorem}
\begin{remark}
Note that Corollary \ref{cor:NewNoise} implies that adding more fluctuating modes to the magnetic field does not alter the result (although it may affect the constant $C$). 
\end{remark}

To fit the notations for the general framework, we denote in the following sections $(x_1, x_2, x_3) = (x, y, \theta) \in \mathbb T^3$, and $(v_1, v_2, v_3)$ the projective coordinates on $\mathbb S^3$. 

\subsection{A remark on the \`a la Furstenberg invariance principle}

One interesting corollary of Theorem \ref{maintheorem} is that positive Lyapunov exponent can be verified for all $\epsilon>0$ by checking a condition only at $\eps = 1$ for non-dissipative systems.
Here, we outline a classical method for inferring positive Lyapunov exponents for volume-preserving SDEs. The approach, originating from Furstenberg \cite{F63}, has been adapted into different versions by various authors, with \cite{B89,C87} being particularly relevant. The version we present here is from \cite{C87}. 

Consider an SDE on an $n$-dimensional connected, compact Riemannian manifold $M$,
\begin{equation*}
dx_t = X_0(x_t)\,dt + \sum_{k=1}^r X_k(x_t)\circ dW^k_t.
\end{equation*}
It is volume preserving if and only if $\mbox{div}X_0 = \mbox{div}X_1 = \dotsc = \mbox{div}X_r = 0$. Under the volume-preserving hypothesis, we can consider an induced process on the special linear bundle $SLM\rightarrow M$. For each $x\in M$, one may consider an element in $SL_xM$ as an ordered basis $\{u_1,\dotsc, u_n\}\subseteq T_xM$ such that the parallelepiped spanned by these vectors has volume $1$. The induced process on $SLM$ is
\begin{align*}
dx_t &= X_0(x_t)\,dt + \sum_{k=1}^r X_k(x_t)\circ dW^k_t \\[3pt]
du_{j,\,t} &= \nabla_{u_{j,\,t}}X_0(x_t)\,dt + \sum_{k=1}^r \nabla_{u_{j,\,t}}X_k(x_t)\circ dW^k_t \quad\mbox{for}\quad j=1,\dotsc, n,
\end{align*}
where $\nabla Y$ denote the covariant derivative of a smooth vector field $Y$ on $M$. 

A sufficient condition for a positive Lyapunov exponent is given in \cite{C87} as the following:
\begin{enumerate}
\item[($H_1'$)] $\mbox{Lie}_x(X_0;\,X_1,\dotsc, X_r) = T_xM$ for all $x\in M$;

\item[($H_2'$)] $\mbox{Lie}_u(Y_0;\,Y_1,\dotsc, Y_r) = T_uSLM$ for some $u\in SLM$.
\end{enumerate}
Condition $(H_1')$ is the standard parabolic H\"{o}rmander condition for the base process. Condition $(H_2')$ ensures that the support of the induced process on $SLM$ with initial $u$ will have a non-empty interior, which is crucial for applying Furstenberg's argument to conclude $\lambda_1>0$. For further details, readers can refer to \cite{C87}. 
Consequently, if \eqref{bprocess} satisfies conditions $(H)$, $(I)$, and $(H_2')$, then not only is $\lambda_1^\eps > 0$ for all $\eps > 0$ (already a consequence of \cite{C87}), but we also obtain the slightly more quantitatve estimate: $\exists C > 0$ such that
\begin{align*}
\liminf_{\eps \to 0} \frac{\lambda_1^\eps}{\eps} \geq C.
\end{align*} 

The structure of the rest of the paper is as follow: Section~\ref{sec:Preliminaries} revisits basic properties of Lyapunov exponents, provides some notes on hypoellipticity and irreducibility, and reviews certain results from \cite{BBPS22}. The proof of Theorem~\ref{maintheorem} is presented in Section~\ref{sec:Mainresult}, followed by the proof of Theorem~\ref{application} in Section~\ref{sec:Application}. A discussion on the sharpness of the decay bound in Theorem~\ref{maintheorem} is provided in Section~\ref{sec:AppendixA}.

\section{Preliminaries}\label{sec:Preliminaries}

\subsection{Lyapunov exponents} \label{sec:LEprelims} 

The existence of the Lyapunov exponents is guaranteed by multiplicative ergodic theorem \cite{B89,O68,R79,W93}; in fact, there exists an almost surely constant Lyapunov spectrum when one has ergodicity.



\begin{theorem}[multiplicative ergodic theorem]\label{introductionthm1}
Let $\mu$ be an ergodic stationary measure for $x_t$ and suppose that 
\begin{align*}
\EE \int \sup_{t \in (0,1)} \left(\log^+\abs{D_x\phi_{t,\omega}} + \log^+\abs{(D_x\phi_{t,\omega})^{-1}} \right) \dee x < \infty.
\end{align*}
Then there is a deterministic number $l \leq d$, a $\mathbb P \times \mu$-measurably-varying, almost-surely invariant flag of linear subspaces $\set{F_{j}}_{j=1}^{l+1}$ such that for almost-every $(\omega, x)$, 
\begin{equation*}
\emptyset = F_{l+1}(\omega,x) \subset F_l(\omega,x) \dotsc, \subset F_1(\omega,x) = T_xM,
\end{equation*}
and a set of deterministic numbers $-\infty < \lambda_l<\dotsc<\lambda_1$ so that $\mathbb{P}\times\mu$-a.e.
\begin{equation*}
\lambda(\omega,x,v) = \lim_{t\rightarrow\infty} \frac{\Vert D_x\phi_{t,\omega}\,v\Vert}{\Vert v\Vert} =  \lambda_k \quad\mbox{for}\quad v\in F_k(\omega, x)\setminus F_{k+1}(\omega,x).
\end{equation*}
Let $d_k = \mbox{dim}F_k(\omega, x) - \mbox{dim}F_{k+1}(\omega, x)$. 
\end{theorem}

Our assumptions that $M$ is compact and the vector fields are smooth imply the integrability conditions required for Theorem~\ref{introductionthm1}; refer to \cite{B89}. 
Here, $\lambda_1$ represents the top Lyapunov exponent, which quantifies the typical stability of the system. Let $\lambda_{\Sigma} = \sum_{k=1}^l d_k\lambda_k$ denote the sum of the Lyapunov spectrum. We define a system as \emph{non-dissipative} if $\lambda_{\Sigma}=0$. 

Next, we review some key results from \cite{BBPS22}, where the authors proposed a Fisher information scheme for studying the small noise asymptotics of the top Lyapunov exponent. 
Let $\rho^{\epsilon}$ and $f^{\epsilon}$ denote the stationary density for $x_t^{\epsilon}$ and $w_t^{\epsilon}$, respectively.
The existence of such stationary distributions is established using the Krylov-Bogoliubov procedure, and is guaranteed when $M$ is compact.
The uniqueness and the smoothness of the densities are deduced from irreducibility and H\"{o}rmander's theorem, utilizing the Doob-Khasminskii theorem \cite{DPZ96}.
Under assumptions (H) and (I), the stationary densities are everywhere positive \cite{HM15}.
Define the Fisher information 
\begin{align*}
FI(\rho^{\epsilon}) &= \frac{1}{2}\sum_{k=1}^r \int_M \frac{\vert X_k^*\rho^{\epsilon}\vert^2}{\rho^{\epsilon}}\,dx, \\
\widetilde{FI}(f^{\epsilon}) &= \frac{1}{2}\sum_{k=1}^r \int_{SM} \frac{\vert \widetilde{X}_k^*f^{\epsilon}\vert^2}{f^{\epsilon}}\,dw.
\end{align*}
A key observation in \cite{BBPS22} is the following formula relating the Lyapunov exponents to the  Fisher information.
This formula can be proven either through the Furstenberg-Khasminskii formula or through the relative entropy formulas in \cite{B89}.

\begin{theorem}[\cite{BBPS22}\label{preliminarythm1} Proposition 3.2 (Fisher Information Identity)]
Under the assumption that $(M,g)$ is a connected, compact Riemannian manifold, and under the assumptions (H) and (I), the following formulas hold:
\begin{align*}
FI(\rho^{\epsilon}) &= –\frac{\lambda_{\Sigma}^{\epsilon}}{\epsilon}, \\[5pt]
\widetilde{FI}(f^{\epsilon}) &= \frac{n\lambda_1^{\epsilon} - 2\lambda_{\Sigma}^{\epsilon}}{\epsilon}. 
\end{align*}
\end{theorem}

An immediate consequence of Theorem~\ref{preliminarythm1} is an infinitesimal characterization of the positive top Lyapunov exponent.
For non-dissipative systems ($\lambda_{\Sigma}^{\epsilon}=0$), Theorem~\ref{preliminarythm1} implies 
\begin{equation*}
0 = n\lambda_1^{\epsilon} = \epsilon \widetilde{FI}(f^{\epsilon}) \quad\Leftrightarrow\quad \widetilde{X}_k^*f^{\epsilon} = 0 \quad\mbox{for}\quad k= 1,\dotsc, r.
\end{equation*}
Furthermore, if we make use of the Kolmogorov equation,
\begin{equation*}
0 = \mathcal{L}_{\epsilon}^* f^{\epsilon} = \tilde{X}_0^*f^{\epsilon} + \frac{\epsilon}{2}\sum_{k=1}^r \left(\tilde{X}_k^*\right)^2f^{\epsilon} = \tilde{X}_0^*f^{\epsilon}.
\end{equation*}
We can then combine the above to obtain
\begin{equation}\label{preliminaryeq1}
0 = \lambda_1^{\epsilon} \quad\Leftrightarrow\quad \widetilde{X}_k^*f^{\epsilon} = 0 \quad\mbox{for}\quad k= 0,1,\dotsc, r.
\end{equation}
In other words, $f^{\epsilon}$ has to be invariant under the flow generated by $X_k$ for $k=0,1,\dotsc, r$. This provides a convenient method to check the positivity of the top Lyapunov exponent; refer to  Section~\ref{sec:Application}. This observation also forms the core of the proof of Theorem~\ref{maintheorem}. 

An important result in \cite{BBPS22} is the following dichotomy, which applies to \eqref{bprocess}:

\begin{theorem}[\cite{BBPS22} Proposition 6.1]\label{preliminarythm2}
Under the assumption that $(M,g)$ is a connected, compact Riemannian manifold, and under the assumptions (H) and (I), the following formulas hold:
\begin{enumerate}
\item Either
\[
\lim_{\epsilon\rightarrow 0}\frac{\lambda_1^{\epsilon}}{\epsilon} = \infty;
\]
\item Or the zero-noise flow $w_t^0$ admits a stationary density $f^0$, and a subsequence of $f^{\epsilon}$ converges to $f^0$ in $L^1$.
\end{enumerate}
\end{theorem}

\subsection{Hypoellipticity and irreducibility of diffusion processes}
In this section, we include some discussion about hypoellipticity and irreducibility of diffusion processes and its consequences. 

Let us start with some general remarks on hypoellipticity and irreducibility. We introduce a more adaptable condition for hypoellipticity based on results from \cite{OR73}. We will be working with a diffusion process on a manifold $N$,
\begin{equation}\label{noteseq1}
dy_t = Y_0(y_t)\,dt + \sum_{k=1}^r Y_k(y_t)\circ dW_t^k.
\end{equation}
In our context, $N$ can be either the compact base manifold $M$ or its sphere bundle $SM$. Let us start by listing several fundamental concepts. Let $\mathcal{B}$ be the Borel $\sigma$-algebra, and $(N,\mathcal{B},\lambda)$ be the measure space with Riemannian volume measure $\lambda$.
Denote $B_b(N, \mathcal{B})$ as the set of bounded Borel functions, and $\mathcal{M}(N,\mathcal{B})$ as the set of Borel probability measures. Let $P_t(y, dz)$ be the transition kernel associated with the diffusion process. The diffusion process is linked to a Markov semigroup on bounded Borel functions $P_t: B_b(N,\mathcal{B})\rightarrow B_b(N,\mathcal{B})$ given by
\begin{equation*}
P_tf(y) = \int f(z)\,P_t(y,\,dz), 
\end{equation*}
with the corresponding dual action on the set of probability measures $P_t^*: \mathcal{M}(N,\mathcal{B})\rightarrow \mathcal{M}(N,\mathcal{B})$,
\begin{equation*}
P_t^*\mu(A) = \int_A \mu(dy)\,P_t(y,\,dz). 
\end{equation*}
For SDEs in Stratonavich form, the generator is 
\begin{equation*}
\mathcal{L} = Y_0 + \frac{1}{2}\sum_1^r Y_k^2,
\end{equation*}
and the adjoint operator is
\begin{equation*}
\mathcal{L}^* = Y_0^* + \frac{1}{2}\sum_1^r \left(Y_k^* \right)^2,
\end{equation*}
where $Y_l^*= -Y_l - \operatorname{div}Y_l$ for $l=0,1,\dotsc, r$.

Smoothing properties of the Markov semigroups are typically obtained by using H\"{o}rmander's hypoelliptic regularity theory (or parabolic regularity if the generator is elliptic).  
As mentioned in the previous section, hypoellipticity hinges on a condition known as the parabolic H\"{o}rmander condition, which concerns the vector fields $\{Y_0; Y_1,\dotsc, Y_r\}$. Note that this condition is equivalent to verifying the elliptic H\"{o}rmander's condition for $\{\partial_t-Y_0,\,Y_1,\dotsc,Y_r\}$ in $T\times N$, where $T$ denotes some open time interval; i.e., the Lie algebra generated by $\{\partial_t-Y_0,\,Y_1,\dotsc,Y_r\}$ spans the tangent space of $T\times N$ at each point. Indeed, it is clear that we can only get the vector $\partial_t$ from $\partial_t-Y_0$; hence, in order to verify the elliptic spanning condition, we need that the Lie algebra generated by $\{[\partial_t-Y_0, Y_k],\, Y_k\}_{k=1}^r$ spans the tangent space of $N$. However,
\begin{equation*}
[\partial_t-Y_0, Y_k] = -[Y_0, Y_k] \quad\mbox{for}\quad k=1,\dotsc, r,
\end{equation*}
we see it is equivalent to verify the parabolic spanning condition in $N$. In particular, the following are equivalent
\begin{enumerate}
\item $\{Y_0; Y_1,\dotsc, Y_r\}$ satisfies the parabolic H\"{o}rmander condition.
\item $\{\partial_t-Y_0,\,Y_1,\dotsc,Y_r\}$ satisfies the elliptic H\"{o}rmander condition.
\end{enumerate}

Hypoellipticity is usually checked by verifying the parabolic H\"{o}rmander condition at every point of the manifold. However, if the coefficients are analytic, then it suffices to check the conditions only on a certain subset of points; this observation is based on the following theorem:

\begin{theorem}[\cite{OR73} Lemma 2.8.1]\label{preliminarythm3}
Let $\{Y_1,\dotsc, Y_r\}$ be a collection of analytic vector fields on an open subset $U\subseteq \mathbb{R}^d$. If there exists $x_0\in U$ such that 
\begin{equation*}
\operatorname{Lie}_{x_0}\{Y_1,\dotsc, Y_r\} \simeq \mathbb{R}^k,
\end{equation*}
for some $1\leq k< d$, then there exists a neighborhood $\mathcal{V}$ of $x_0$ and a $k$-dimensional analytic manifold $V$ defined in $\mathcal{V}$ such that $x_0\in V$ and for all $x$
\begin{equation*}
\operatorname{Lie}_x\{Y_1,\dotsc, Y_r\} \subseteq T_xV, \quad\forall x\in V \cap \mathcal{V}.
\end{equation*}
\end{theorem}

To apply Theorem~\ref{preliminarythm3} for studying parabolic H\"{o}rmander condition in a general setting, we recall some notions in geometry for the readers' convenience.  An analytic manifold is a smooth manifold such that the transition maps  are real analytic. Let  $\{(U_i,\,\phi_i)\}_{i\in I}$ be the charts defining the smooth structure on $N$ for some index set $I$, and let $\pi: TN\rightarrow N$ be the canonical projection such that 
\begin{equation*}
\pi(p,v) = p \quad\mbox{for}\quad (p,v)\in T_pN \quad\mbox{and}\quad p\in N.
\end{equation*} 
One can define a natural smooth structure on the tangent bundle $TN$ through the charts $\psi_i: \pi^{-1}(U)\rightarrow U\times \mathbb{R}^n$
\begin{equation*}
\psi_i: \left . \sum_{j=1}^n v^j\frac{\partial}{\partial y^j}\right\vert_p \mapsto (y^1(p),\dotsc, y^n(p), v^1, \dotsc, v^n) \quad\forall p \in N.
\end{equation*}
In this setting, a vector field $Y: N\rightarrow TN$ is a section such that $\pi\circ Y = id_N$, and it is smooth if $Y$ is a smooth map between $N$ and $TN$. Similarly, if $N$ is an analytic manifold, then so is $TN$, and the notion of analytic vector fields is then defined similarly. For more details, readers can refer to \cite{L12}, \cite{V84}.

Using Lemma \ref{preliminarythm3}, a more adaptable H\"{o}rmander condition is formulated below:
\begin{enumerate}
\item[(H1)] At each point $y\in N\setminus B$, the vector fields $\{Y_0; Y_1,\dotsc, Y_r\}$ satisfy the parabolic H\"{o}rmander condition, where $B$ is a set of points which locally lies on some $d-1$-dimensional smooth sub-manifold.

\item[(H2)] At each point $y\in B$, there exists a local chart $(\varphi,\,U)$ such that  
\begin{equation*}
\varphi(U\cap B) \subseteq  O,
\end{equation*} 
where $O$ is a $d-1$-dimensional smooth submanifold described by the zero set of some smooth function $\Phi(y_1,\dotsc, y_d)=0$ in the local coordinate. Moreover, there exists some $k=0,1\dotsc,r$ such that $Y_k\Phi\neq 0$.  
\end{enumerate}

\begin{corollary}\label{preliminarycor1}
Let $N$ be an analytic manifold. If $\{Y_0; Y_1,\dotsc, Y_r\}$ is a collection of analytic vector fields on $N$, then the following are equivalent:
\begin{enumerate}
\item The collection $\{Y_0; Y_1,\dotsc, Y_r\}$ satisfies (H1) and (H2).

\item The collection $\{Y_0; Y_1,\dotsc, Y_r\}$ satisfies the parabolic H\"{o}rmander condition. 
\end{enumerate}
\end{corollary}

\begin{proof}
Parabolic H\"{o}rmander condition implies (H1) and (H2) is immediate, so we focus on the other direction. As we discussed in the beginning of this section, the collection $\{Y_0; Y_1,\dotsc, Y_r\}$ satisfying the parabolic H\"{o}rmander condition on $N$ is equivalent to $\{\partial_t - Y_0, Y_1, \dotsc, Y_r\}$ satisfying the elliptic H\"{o}rmander condition on $T\times N$. For convenience, we denote $Z_0 = \partial_t - Y_0$ znd $Z_k = Y_k$ for $k=1,\dotsc, r$.

Given $p_0\in B$, there exists a neighborhood $U\subseteq N$ of $p_0$, and a coordinate chart $\psi: T\times U \rightarrow T\times\mathbb{R}^d$ defined as
$\psi(t, p) = (t, \varphi(p))$
for $(t, p)\in T\times U$
such that
\begin{equation*}
\psi\left(T\times \{p_0\}\right) \subseteq T\times \left\{ y \in \varphi(U): \Phi(y) = 0\right\},
\end{equation*}
where $\Phi: \mathbb{R}^d\rightarrow\mathbb{R}$ is the smooth function described in (H2). By defining $\Psi: \psi(T\times U)\rightarrow \mathbb{R}$ via
\begin{equation*}
\Psi(t, y) = \Phi(y),
\end{equation*}
we have that
\begin{equation*}
T\times \left\{ y \in \varphi(U): \Phi(y) = 0\right\} = \left\{(t,y)\in \psi(T\times U): \Psi(t, y) = 0\right\}.
\end{equation*}

Suppose by contradiction that
\begin{equation*}
\operatorname{Lie}_{(t_0, y_0)}\left\{ Z_0, Z_1, \dotsc, Z_r\right\} \simeq \mathbb{R}^k, \quad 1\leq k < d+1
\end{equation*}
for some $(t_0,y_0)\in \psi\left(T\times\{p_0\}\right)$; here, we abuse notations and use the same symbols for the vector fields in local coordinates. By Theorem~\ref{preliminarythm3}, there exists a $k$-dimensional analytic manifold $V\subseteq \psi\left( T\times U\right)$ such that $(t_0,y_0)\in V$, and
\begin{equation}\label{preliminaryeq3}
\operatorname{Lie}_{(t, y)}\left\{ Z_0, Z_1, \dotsc, Z_r\right\} \subseteq T_{(t,y)}V, \quad\forall (t,y)\in V.
\end{equation}
By (H2), there exists a vector field $Y_k$, $k\in\{0,\dotsc, r\}$, such that $Y_k\Phi(t_0, y_0)\neq 0$, which further implies that
\begin{equation*}
Z_k\Psi(t_0,y_0) = Y_k\Phi(y_0) \neq 0.
\end{equation*}
Hence, the integral curve defined by the following initial value problem:
\begin{align*}
\dot{z}(s) &= Z_k(z(s)), \\[3pt]
z(0) &= (t_0,y_0)
\end{align*}
cannot lie in $\left\{\Psi(t, y) = 0\right\}$; in particular, there exists $\epsilon_1> 0$ such that for all $0< s < \epsilon_1$, $z(s) \not\in \left\{\Psi(t,y) = 0\right\}$. By (\ref{preliminaryeq3}), there exists $\epsilon_2>0$ such that for all $0< s< \epsilon_2$, $z(s) \in V$. Hence, 
\begin{equation*}
z(s) \in V \cap \left\{ \Psi(t,y) = 0\right\}^c
\end{equation*}
for $0 < s < \min\{\epsilon_1, \epsilon_2\}$. However, by (H1)
\begin{equation*}
\operatorname{Lie}_z\left\{ Z_0, Z_1, \dotsc, Z_r\right\} \simeq \mathbb{R}^{d+1} \quad\forall z \in \left\{ \Psi(t,y) = 0\right\}^c,
\end{equation*}
this leads to a contradiction.
\end{proof}

To illustrate the use of Corollary~\ref{preliminarycor1}, let us consider a family of vector fields $\{Y_0; Y_1\}$ in $\mathbb{R}^2$, where $Y_0 = y^n\partial_x$ and $Y_1 = \partial_y$ for some $n\in\mathbb{N}$. A routine calculation shows that 
\begin{equation*}
\left[Y_1,Y_0\right] = ny^{n-1}\partial_x,
\end{equation*}
and we observe that $\left\{Y_1, [Y_1, Y_0]\right\}$ spans $T_{(x,y)}\mathbb{R}^2$ except when $y=0$. To explicitly verify that $\{Y_0; Y_1\}$ satisfies the parabolic H\"{o}rmander condition in $\mathbb{R}^2$, one needs to compute at least $n$ brackets. However, by using conditions (H1) and (H2), one can see that the parabolic H\"{o}rmander condition is true immediately after computing $[Y_1,Y_0]$. Adopting the notation from (H1) and (H2),
\begin{equation*}
B = \left\{ (x,y)\in\mathbb{R}^2: y = 0\right\},
\end{equation*}
and thus $B$ lies within the zeroes of $\Phi(x,y) = y$. Since $(Y_1\Phi)(x, 0) = 1 \neq 0$, conditions (H1) and (H2) are satisfied. Consequently, we conclude that the family $\{Y_0; Y_1\}$ satisfies the parabolic H\"{o}rmander condition by Corollary~\ref{preliminarycor1}.

Irreducibility is usually reduced to studying certain control problems via the Stroock-Varadhan support theorem \cite{SV72}. Before we proceed, let us establish some terminologies from geometric control theory following \cite{J97}. Given a family of vector fields $\mathcal{F}$, a continuous curve $\eta: [0,t]\rightarrow N$ is an integral curve of $\mathcal{F}$, if there exists a partition $0= t_0 < t_1 <\dotsc<t_m = t$ and vector fields $V_1,\dotsc, V_m$ in $\mathcal{F}$ such that
\begin{equation*}
\dot{\eta}(t) = V_i(\eta(t)) \quad\mbox{for}\quad t_{i-1} < t < t_i, \quad i=1,\dotsc, m.
\end{equation*}
We denote $\mathcal{A}_{\mathcal{F}}(y, t)$  the reachable set of $\mathcal{F}$ with initial $y$ at time $t$, to be the union of all the values of the integral curves of $\mathcal{F}$ starting from $y$ at time $t$. We also define
\begin{equation*}
\mathcal{A}_{\mathcal{F}}(y, \leq T) = \bigcup_{t\leq T}\mathcal{A}_{\mathcal{F}}(y, t) \quad\mbox{and}\quad \mathcal{A}_{\mathcal{F}}(B, t) = \bigcup_{y\in B} \mathcal{A}_{\mathcal{F}}(y, t) \quad\mbox{for}\quad B\subseteq N.
\end{equation*} 
For the accessible set of the control problem associated with the diffusion process by the Stroock-Varadhan support theorem, we have the family of vector fields 
\begin{equation*}
\mathcal{F}_0=Y_0+\mathcal{Y} \quad\mbox{and}\quad \mathcal{Y}=\mbox{span}\{Y_1,\dotsc, Y_r\}.
\end{equation*} 
Let
\begin{equation*}
\phi_{t,y}^{\alpha} = y + \int_0^t Y_0(\phi_{s,y}^{\alpha})\,ds + \sum_{k=1}^r\int_0^t Y_k(\phi_{s,y}^{\alpha})\,\alpha_k(s)ds,
\end{equation*}
where $\mathbf{\alpha}(s)=(\alpha_1,\dotsc,\alpha_r):\mathbb{R}\rightarrow \mathbb{R}^r$ is a piecewise constant control with finite partition. The reachable set is then
\begin{equation*}
\mathcal{A}_{\mathcal{F}_0}(y, t) = \bigcup_{\alpha} \phi_{t,y}^{\alpha}.
\end{equation*}

\begin{theorem}[Stroock-Varadhan support theorem; \cite{SV72}]
If there exists $t>0$ such that 
\begin{equation*}
\mathcal{A}_{\mathcal{F}_0}(y, t) = N \quad\forall y\in N,
\end{equation*}
then for each $y\in N$ and every $U$ open in $N$,
\begin{equation*}
P_t(y,\, U)>0.
\end{equation*}
\end{theorem}

An important consequence of $\{Y_0;\,Y_1,\dotsc, Y_r\}$ satisfying the parabolic H\"ormander condition is that the accessible set $\mathcal{A}_{\mathcal{F}_0}(y, t)$ has a non-empty interior for $t>0$; see Proposition~5.10 of \cite{BBPS22} and Theorem~3 in Section~3.1 of \cite{J97}. By an argument as in the proof of Theorem~13 in Section~3.4 of \cite{J97}, we then only need to check 
\begin{equation}\label{preliminaryeq2}
\overline{\mathcal{A}_{\mathcal{F}_0}(y, t)} = N \quad\forall y\in N.
\end{equation}

If there are certain directions along which the speed of the process is uniformly bounded, condition (\ref{preliminaryeq2}) can be tricky to verify. This is the case for our charged particles example, since the magnetic field does not change the kinetic energy. 
To verify the irreducibility for the associated projective process, we use a two-step strategy to separate the control problems for the base manifold and the unit tangent space; see Section~\ref{irreducibility} for details.

\section{Proof of Theorem~\ref{maintheorem} and Corollary~\ref{cor:NewNoise}}\label{sec:Mainresult}
In this section, we present the proof for the quantitative dichotomy, Theorem~\ref{maintheorem}, where uniqueness of the stationary density plays an important role. 

\begin{proof}
 Regarding the dichotomy Theorem~\ref{maintheorem} of the parameterized SDEs, we first show that if $\lambda_1^{\epsilon_0}=0$ for some $\epsilon_0>0$, then $\lambda_1^{\epsilon}=0$ for all $\epsilon>0$. Indeed, via (\ref{preliminaryeq1}), 
\begin{equation*}
0 =  \tilde{X}_k^*f^{\epsilon_0} \quad\mbox{for}\quad k=0,1,\dotsc, r.
\end{equation*}
In particular, this shows
\begin{equation*}
0 = \tilde{X}_0^*\,f^{\epsilon_0} + \frac{\epsilon}{2}(\tilde{X}_k^*)^2\,f^{\epsilon_0} \quad\forall \epsilon>0,
\end{equation*}
which due to unique ergodicity implies
\begin{equation*}
f^{\epsilon_0} = f^{\epsilon} \quad\forall \epsilon>0.
\end{equation*}
Therefore,
\begin{equation*}
\tilde{X}_k^*f^{\epsilon} = \tilde{X}_k^*f^{\epsilon_0} = 0, \quad \epsilon>0 \quad\mbox{and}\quad k=0,1,\dotsc, r \quad\Rightarrow\quad \lambda_1^{\epsilon}=0 \quad\forall \epsilon>0.
\end{equation*}
This shows that $\lambda_1^{\epsilon_0} = 0$ for some $\epsilon_0>0$ implies $\lambda_1^{\epsilon} = 0$ for all $\epsilon>0$. 

Now suppose $\lambda_1^{\epsilon}>0$ for all $\epsilon>0$, but that 
\begin{equation*}
\liminf_{\epsilon\rightarrow 0} \frac{\lambda_1^{\epsilon}}{\epsilon} = 0.
\end{equation*}
Then we may pick a subsequence $\{\epsilon_n\}_n$ such that $\lim_{n\rightarrow\infty}\epsilon_n = 0$, and
\begin{equation*}
\lim_{n\rightarrow\infty} \frac{\lambda_1^{\epsilon_n}}{\epsilon_n} = 0;
\end{equation*}
our goal is to show that this hypothesis will lead to $\lambda_1^{\epsilon}=0$ for all $\epsilon>0$, a contradiction.

By Theorem~\ref{preliminarythm1}, we have
\begin{equation}\label{mainresulteq1}
0 = \lim_{n\rightarrow\infty} \frac{d\lambda_1^{\epsilon_n}}{\epsilon_n} = \lim_{n\rightarrow\infty}\frac{1}{2}\sum_{k=1}^r  \int_{\mathbb{S}M} \frac{\vert\tilde{X}_k^* f^{\epsilon_n}\vert^2}{f^{\epsilon_n}}\,dw,
\end{equation}
which implies $\tilde{X}_k^*f^{\epsilon_n}/\sqrt{f^{\epsilon_n}}$ converges strongly in $L^2(dw)$ to $0$. Also, Theorem~\ref{preliminarythm2} implies that 
\begin{equation*}
\lim_{n\rightarrow\infty} \int_{\mathbb{S}M} \vert f^{\epsilon_n}-f^0\vert \,dw = 0,
\end{equation*}
for some invariant density of the zero-noise dynamics $f^0\in L^1(dw)$, up to extracting a further subsequence of $\{f^{\epsilon_n}\}$. 

We claim that 
\begin{equation*}
\tilde{X}_k^*f^0 = 0 \quad\forall k=0,1,\dotsc, r,
\end{equation*}
in the weak sense. For $k=0$, it is automatic, since $f^0$ is a stationary density for the zero-noise flow; our focus will be on $k=1,\dotsc, r$. Let $\phi\in C^{\infty}_0(SM)$,
\begin{equation*}
\int_{\mathbb{S}M} \phi\,\tilde{X}_k^*f^0\,dw = \int_{\mathbb{S}M} \phi\,\tilde{X}_k^*(f^0-f^{\epsilon_n})\,dw + \int_{\mathbb{S}M} \phi\,\tilde{X}_k^*f^{\epsilon_n}\,dw.
\end{equation*}
The first term goes to $0$ as $n\rightarrow\infty$ due to the $L^1(dw)$ convergence,
\begin{equation*}
\left\vert\int_{\mathbb{S}M} \phi\,\tilde{X}_k^*(f^0-f^{\epsilon_n})\,dw\right\vert = \left\vert\int_{\mathbb{S}M} \tilde{X}_k\phi\,(f^0-f^{\epsilon_n})\,dw\right\vert \leq \Vert\tilde{X}_k\phi\Vert_{L^{\infty}(dw)}\,\Vert f^0 - f^{\epsilon_n}\Vert_{L^1(dw)}.
\end{equation*}
For the second term. 
\begin{align*}
\left\vert\int_{\mathbb{S}M} \phi\,\tilde{X}_k^*f^{\epsilon_n}\,dw\right\vert = \left\vert\int_{\mathbb{S}M} \phi\,\sqrt{f^{\epsilon_n}}\,\frac{\tilde{X}_k^*f^{\epsilon_n}}{\sqrt{f^{\epsilon_n}}}\,dw\right\vert &\leq \left(\int_{\mathbb{S}M} \phi^2\,f^{\epsilon_n}dw\right)^{1/2}\,\left(\int_{\mathbb{S}M}\frac{\vert\tilde{X}_k^*f^{\epsilon_n}\vert^2}{f^{\epsilon_n}}\,dw\right)^{1/2} \\[3pt]
&\leq \Vert\phi\Vert_{L^{\infty}(dw)}\,\Vert \tilde{X}_k^*f^{\epsilon_n}/\sqrt{f^{\epsilon_n}}\Vert_{L^2(dw)},
\end{align*}
which converges to $0$ due to our hypothesis (\ref{mainresulteq1}).

Finally, we demonstrate how the above leads to a contradiction. Under our hypothesis (\ref{mainresulteq1}), we deduced
\begin{equation*}
\tilde{X}_k^*f^0 = 0 \quad\mbox{for}\quad k=0,1,\dotsc, r,
\end{equation*}
which implies
\begin{equation*}
\mathcal{L}^* f^0 = \tilde{X}_0^*f^0 + \frac{1}{2}\sum_{k=1}^r \left(\tilde{X}_k^*\right)^2f^0 = 0.
\end{equation*}
By unique ergodicity, this implies $f^1=f^0$, and
\begin{equation*}
0 = FI(f^0) = FI(f^1),
\end{equation*}
which is a contradiction. 
\end{proof}

Now we prove Corollary \ref{cor:NewNoise}. 
\begin{proof}
By contradiction, we assume that $\lambda_1(x) > 0$ and $\lambda_1(y) = 0$. Denote $f: \mathbb{S}M\rightarrow \mathbb{R}$ and $g: \mathbb{S}M\rightarrow \mathbb{R}$ the stationary densities for the projective processes associated with $x_t$ and $y_t$ respectively. Using Fisher information formula, one can see that
\begin{align*}
\lambda_1(y) = 0 &\quad\Leftrightarrow\quad \tilde{X}_k^*g = 0 \quad\mbox{for}\quad k = 0,1,\dotsc, r+1. \\[3pt]
&\quad\Rightarrow\quad \tilde{X}_k^*g = 0 \quad\mbox{for}\quad k = 0,1,\dotsc, r. \\[3pt]
&\quad\Rightarrow\quad \left[ \tilde{X}_0^* + \frac{1}{2}\sum_{k=1}^r\left(\tilde{X}_k^*\right)^2\right] g = 0. 
\end{align*}
This shows that $g$ solves the same Kolmogorov equation that $f$ solves; under assumptions (H) and (I), such density is unique, so $f=g$. This contradicts that $\lambda_1(x) >0 $. 
\end{proof}


\section{An application to positive Lyapunov exponent in electrodynamics}\label{sec:Application}
In this section, we apply Theorem~\ref{maintheorem} to prove our result about the electrodynamics under fluctuating magnetic fields, Theorem~\ref{application}. We will show that $\lambda_1^{\epsilon} > 0$ for all $\epsilon>0$, assuming hypoellipticity and irreducibility; the verification of hypoellipticity and irreducibility is postponed to Section~\ref{hypoellipticity} and Section~\ref{irreducibility}. As a consequence, we also get 
\[
\liminf_{\epsilon\rightarrow 0}\frac{\lambda_1^{\epsilon}}{\epsilon} \geq C.
\] 

A calculation shows
\begin{align}
\tilde{X}_0 &= \begin{bmatrix}
\cos{x_3}, & \sin{x_3}, & 1
\end{bmatrix}\cdot\nabla_{x} + \begin{bmatrix}
-v_3\sin{x_3} - \alpha v_1, & v_3\cos{x_3} - \alpha v_2, & -\alpha v_3)
\end{bmatrix}\cdot\nabla_{v}, \label{applicataioneq1} \\[3pt]
\tilde{X}_1 &= \begin{bmatrix}
0, & 0, & 1
\end{bmatrix}\cdot\nabla_{x}, \label{applicataioneq2} \\[3pt]
\tilde{X}_2 &= \begin{bmatrix}
0, & 0, & \sin{x_1}
\end{bmatrix}\cdot\nabla_{x} + v_1\cos{x_1}\,\begin{bmatrix}
-v_3v_1, & -v_3v_2, & 1-(v_3)^2
\end{bmatrix}\cdot\nabla_{v}, \label{applicataioneq3} \\[3pt]
\tilde{X}_3 &= \begin{bmatrix}
0, & 0, & \cos{x_1}
\end{bmatrix}\cdot\nabla_{x} - v_1\sin{x_1}\,\begin{bmatrix}
-v_3v_1, & -v_3v_2, & 1-(v_3)^2
\end{bmatrix}\cdot\nabla_{v}, \label{applicataioneq4} \\[3pt]
\tilde{X}_4 &= \begin{bmatrix}
0, & 0, & \sin{x_2}
\end{bmatrix}\cdot\nabla_{x} + v_2\cos{x_2}\,\begin{bmatrix}
-v_3v_1, & -v_3v_2, & 1-(v_3)^2
\end{bmatrix}\cdot\nabla_{v}, \label{applicataioneq5} \\[3pt]
\tilde{X}_5 &= \begin{bmatrix}
0, & 0, & \sin{x_1}
\end{bmatrix}\cdot\nabla_{x} - v_2\sin{x_2}\,\begin{bmatrix}
-v_3v_1, & -v_3v_2, & 1-(v_3)^2
\end{bmatrix}\cdot\nabla_{v}, \label{applicataioneq6}
\end{align}
where $\alpha = -v_3v_1\sin{x_3} + v_3v_2\cos{x_3}$. Also, $\mbox{div}\tilde{X}_0 = 5v_3(v_1\sin{x_3} - v_2\cos{x_3})$, $\mbox{div}\tilde{X}_1 = 0$, $\mbox{div}\tilde{X}_2 = -5v_3v_1\cos{x_1}$, $\mbox{div}\tilde{X}_3 = 5v_3v_1\sin{x_1}$, $\mbox{div}\tilde{X}_4 = -5v_3v_2\cos{x_2}$ and $\mbox{div}\tilde{X}_5 = 5v_3v_2\sin{x_2}$. By Theorem~\ref{maintheorem}, a necessary and sufficient condition for $\lambda_1^\epsilon=0$ for $\epsilon>0$ is that
\begin{multline*}
\begin{bmatrix}
\cos{x_3}, & \sin{x_3}, & 1
\end{bmatrix}\cdot\nabla_{x}g + \begin{bmatrix}
-v_3\sin{x_3} - \alpha v_1, & v_3\cos{x_3} - \alpha v_2, & -\alpha v_3)
\end{bmatrix}\cdot\nabla_{v}g \\ 
= -5v_3(v_1\sin{x_3} - v_2\cos{x_3}),
\end{multline*}
\begin{align*}
\begin{bmatrix}
0, & 0, & 1
\end{bmatrix}\cdot\nabla_{x}g &= 0, \\[3pt]
\begin{bmatrix}
0, & 0, & \sin{x_1}
\end{bmatrix}\cdot\nabla_{x}g + v_1\cos{x_1}\,\begin{bmatrix}
-v_3v_1, & -v_3v_2, & 1-(v_3)^2
\end{bmatrix}\cdot\nabla_{v}g &= 5v_3v_1\cos{x_1}, \\[3pt]
\begin{bmatrix}
0, & 0, & \cos{x_1}
\end{bmatrix}\cdot\nabla_{x}g - v_1\sin{x_1}\,\begin{bmatrix}
-v_3v_1, & -v_3v_2, & 1-(v_3)^2
\end{bmatrix}\cdot\nabla_{v}g &= -5v_3v_1\sin{x_1}, \\[3pt]
\begin{bmatrix}
0, & 0, & \sin{x_2}
\end{bmatrix}\cdot\nabla_{x}g + v_2\cos{x_2}\,\begin{bmatrix}
-v_3v_1, & -v_3v_2, & 1-(v_3)^2
\end{bmatrix}\cdot\nabla_{v}g &= 5v_3v_2\cos{x_2}, \\[3pt]
\begin{bmatrix}
0, & 0, & \sin{x_1}
\end{bmatrix}\cdot\nabla_{x}g - v_2\sin{x_2}\,\begin{bmatrix}
-v_3v_1, & -v_3v_2, & 1-(v_3)^2
\end{bmatrix}\cdot\nabla_{v}g &= -5v_3v_2\sin{x_2},
\end{align*}
where $g=\log{f}$ and $f$ is the stationary density for the projective process. The system can be further reduced to
\begin{multline*}
\begin{bmatrix}
\cos{x_3} & \sin{x_3} & 0
\end{bmatrix}\cdot\nabla_{x}g + \begin{bmatrix}
-v_3\sin{x_3} - \alpha v_1 & v_3\cos{x_3} - \alpha v_2 & -\alpha v_3)
\end{bmatrix}\cdot\nabla_{v}g \\
= -5v_3(v_1\sin{x_3} - v_2\cos{x_3}), 
\end{multline*}
\begin{align*}
\begin{bmatrix}
0 & 0 & 1
\end{bmatrix}\cdot\nabla_{x}g &= 0, \\[3pt]
\begin{bmatrix}
-v_3v_1 & -v_3v_2 & 1-(v_3)^2
\end{bmatrix}\cdot\nabla_{v}g &= 5v_3,
\end{align*}
since the second equation implies $\partial_{x_3}g = 0$. In particular, if we take $a(x_3) = (0,0,x_3, 1/\sqrt{3}, 1/\sqrt{3}, 1/\sqrt{3})$, then the equations become
\begin{multline}
\begin{bmatrix}\label{cheq1}
\cos{x_3} & \sin{x_3} & 0
\end{bmatrix}\cdot\nabla_{x}g + \frac{1}{\sqrt{3}}\,\begin{bmatrix}
-\sin{x_3} - b(x_3) & \cos{x_3} - b(x_3) & -b(x_3)
\end{bmatrix}\cdot\nabla_{v}g \\
= \frac{5}{3}\,(\cos{x_3} - \sin{x_3}), 
\end{multline}
\begin{align}
\begin{bmatrix}\label{cheq2}
0 & 0 & 1
\end{bmatrix}\cdot\nabla_{x}g &= 0, \\[3pt]
\frac{1}{3}\,\begin{bmatrix}\label{cheq3}
-1 & -1 & 2
\end{bmatrix}\cdot\nabla_{v}g &= \frac{5}{\sqrt{3}},
\end{align}
where $b(x_3) = (\cos{x_3}-\sin{x_3})/3$; moreover, 
\begin{align*}
\partial_{x_3}g = 0 \quad\Rightarrow\quad \partial_{x_3}\nabla g = \nabla\partial_{x_3} g = 0 \quad\Rightarrow\quad (\nabla g)(a(x_3)) = (\nabla g)(a(0)) \quad\mbox{for}\quad x_3\in [0,2\pi).
\end{align*}
Denote $z_0= (x_0, y_0) = (\nabla_{x} g(a(x_3)), \nabla_{v} g(a(x_3)))$ and integrate (\ref{cheq1}) along intervals $[0,\pi]$, $[-\pi/2, \pi/2]$ and $[0,\pi/3]$, 

\begin{align}
\begin{bmatrix}\label{cheq4}
0 & 2 & 0
\end{bmatrix}\cdot x_0 + \frac{1}{3\sqrt{3}}\,\begin{bmatrix}
-4 & 2 & 2
\end{bmatrix}\cdot y_0 &= -\frac{10}{3} \\[3pt]
\begin{bmatrix}\label{cheq5}
2 & 0 & 0
\end{bmatrix}\cdot x_0 + \frac{1}{3\sqrt{3}}\,\begin{bmatrix}
-2 & 4 & -2
\end{bmatrix}\cdot y_0 &= \frac{10}{3} \\[3pt]
\frac{1}{2}\,\begin{bmatrix}\label{cheq6}
\sqrt{3} & 1 & 0
\end{bmatrix}\cdot x_0 + \frac{1}{6\sqrt{3}}\,\begin{bmatrix}
-2-\sqrt{3} & 1+ 2\sqrt{3} & 1-\sqrt{3}
\end{bmatrix}\cdot y_0 &= \frac{5}{6}(\sqrt{3} -1).
\end{align}
Then (\ref{cheq3}), (\ref{cheq4}), (\ref{cheq5}) and (\ref{cheq6}) shows that $z_0$ solves a linear system $Az = b$, where
\begin{align*}
A = \begin{bmatrix}
0 & 0 & 0 & -\frac{1}{3} & -\frac{1}{3} & \frac{2}{3} \\[5pt]
0 & 2 & 0 & -\frac{4}{3\sqrt{3}} & \frac{2}{3\sqrt{3}} & \frac{2}{3\sqrt{3}} \\[5pt]
2 & 0 & 0 & -\frac{2}{3\sqrt{3}} & \frac{4}{3\sqrt{3}} & -\frac{2}{3\sqrt{3}} \\[5pt]
\frac{\sqrt{3}}{2} & \frac{1}{2} & 0 & -\frac{2+\sqrt{3}}{6\sqrt{3}} & \frac{1+2\sqrt{3}}{6\sqrt{3}} & \frac{1-\sqrt{3}}{6\sqrt{3}}
\end{bmatrix}
\quad\mbox{and}\quad
b = \begin{bmatrix}
\frac{5}{\sqrt{3}} \\[5pt]
-\frac{10}{3} \\[5pt]
\frac{10}{3} \\[5pt]
\frac{5(\sqrt{3}-1)}{6}
\end{bmatrix}.
\end{align*}
However, using Gaussian elimination, we have the row equivalence,
\begin{align*}
\begin{bmatrix}
A & b
\end{bmatrix} \sim C = \begin{bmatrix}
* & * & * & * & * & * & * \\[5pt]
* & * & * & * & * & * & * \\[5pt]
* & * & * & * & * & * & * \\[5pt]
0 & 0 & 0 & 0 & 0 & 0 & 1 \\[5pt]
\end{bmatrix},
\end{align*}
where $*$ denotes unimportant entries; this implies that the linear system $Az=b$ is inconsistent, and therefore such $z_0$ cannot exist, a contradiction to our hypothesis that $\lambda_1^{\epsilon} = 0$ for $\epsilon>0$; hence, our proof is finished.  
\subsection{Hypoellipticity} \label{hypoellipticity}

In this section, we check the parabolic H\"{o}rmander condition (H) for the projective process $w_t$; the parabolic H\"{o}rmander condition for the base process $x_t$ is clear, so we omit it. Specifically, we showed the following
\begin{theorem}\label{hypoellipticitythm1}
For $w$ in $\mathbb{ST}^3$ such that $v_3\neq\pm 1$, the following collection spans $T_{w}\mathbb{ST}^3$
\begin{equation*}
\tilde{X}_1,\,\tilde{X}_k,\, \big[\tilde{X}_1,\tilde{X}_0\big],\, \big[\tilde{X}_1,[\tilde{X}_1,\tilde{X}_0]\big],\, \big[\tilde{X}_k,[\tilde{X}_1, \tilde{X}_0]\big],\, \big[\tilde{X}_k,[\tilde{X}_1[,\tilde{X}_1, \tilde{X}_0]]\big], \qquad k=2,\dotsc,5.
\end{equation*}
\end{theorem}


Before we proceed, let us explain how this is enough to show hypoellipticity. To check (H1), we see that Theorem~\ref{hypoellipticitythm1} implies that the spanning condition is checked except at
\begin{equation*}
B = \left\{w\in\mathbb{ST}^3: v_3 = \pm 1\right\},
\end{equation*}
and $B$ is contained in the hypersurfaces defined by $\left\{ v_1=0\right\}$ and $\left\{ v_2=0\right\}$. Denote $\Phi_1(x,v) = v_1$ and $\Phi_2(x,v) = v_2$. Then
\begin{align*}
\tilde{X}_0\Phi_1 &= -v_3\sin{x_3} - \alpha v_1 \\[3pt]
\tilde{X}_0\Phi_2 &= v_3\cos{x_3} - \alpha v_2,
\end{align*}
where $\alpha = -v_3v_1\sin{x_3} + v_3v_2\cos{x_3}$.
In particular,
\begin{equation*}
\left[\tilde{X}_0\Phi_1(x_1,x_2,x_3,0,0,\pm 1)\right]^2 + \left[\tilde{X}_0\Phi_2(x_1,x_2,x_3,0,0,\pm 1)\right]^2 = \sin^2{x_3} + \cos^2{x_3} = 1. 
\end{equation*}
This implies at least one of $\tilde{X}_0\Phi_1(x_1,x_2,x_3,0,0,\pm 1)$ and $\tilde{X}_0\Phi_2(x_1,x_2,x_3,0,0,\pm 1)$ is non-vanishing, which implies (H2). By Corollary~\ref{preliminarycor1}, Theorem~\ref{hypoellipticitythm1} is enough to show that the parabolic H\"{o}rmander condition holds everywhere.

\subsubsection{Proof of Theorem~\ref{hypoellipticitythm1}}
We verify the spanning condition for $\{ v_3\neq\pm 1\}$ by considering a moving frame (a collection of vector fields which form an orthonormal basis at each point), 
\begin{align*}
e_1 &= \begin{bmatrix}
1 & 0 & 0
\end{bmatrix}\cdot\nabla_{x} \\[3pt]
e_2 &= \begin{bmatrix}
0 & 1 & 0
\end{bmatrix}\cdot\nabla_{x} \\[3pt]
e_3 &= \begin{bmatrix}
0 & 0 & 1
\end{bmatrix}\cdot\nabla_{x} \\[3pt]
e_4 &= \frac{1}{\sqrt{1-( v_3)^2}}\begin{bmatrix}
- v_3 v_1 & - v_3 v_2 & 1-( v_3)^2
\end{bmatrix}\cdot\nabla_{v} \\[3pt]
e_5 &= \frac{1}{\sqrt{1-( v_3)^2}}\begin{bmatrix}
- v_2 &  v_1 & 0
\end{bmatrix}\cdot\nabla_{v}.
\end{align*}
One can see that $e_4$ and $e_5$ is not defined at $ v_3=\pm 1$; therefore, this moving frame is only valid for $ v_3\neq\pm 1$. We record a basic formula, and the brackets of the basis.

\begin{lemma}
Let $a$, $b$ be smooth functions, then
\begin{equation}\label{hypoellipticityeq7}
\left[a\,e_i, b\,e_j\right] = -\left(b\,e_ja\right)\,e_i + \left(a\,e_ib\right)\,e_j + ab\,[e_i, e_j]; 
\end{equation}
moreover, 
\begin{equation*}
[e_i,e_j] = 0
\end{equation*}
for $(i,j)\not\in \{ (4,5), (5,4)\}$, and
\begin{equation*}
[e_4,e_5] = \frac{ v_3}{\sqrt{1-( v_3)^2}}\,e_5.
\end{equation*}
\end{lemma}
\begin{proof}
For (\ref{hypoellipticityeq7}), note that  
\begin{align*}
[a\,e_i, b\,e_j] &= \left(a\,e_ib\right)\,e_j + b\left(a\,e_i\right)e_j - \left(b\,e_ja\right)\,e_i - a\left(b\,e_j\right)\,e_i \\[5pt]
&= - \left(b\,e_ja\right)\,e_i + \left(a\,e_ib\right)\,e_j + ab\,[e_i, e_j]. 
\end{align*}
For the commutators of the basis, it is clear that $[e_i,e_j]=0$ for $(i,j)\neq (4,5)$; for $(i,j)=(4,5)$, 
\begin{align*}
    e_4\,e_5 &= \begin{bmatrix}
    - v_3 v_1 & - v_3 v_2 & 1-( v_3)^2
    \end{bmatrix}
    \begin{bmatrix}
    0 & \frac{1}{1-( v_3)^2} & 0 \\[5pt]
    -\frac{1}{1-( v_3)^2} & 0 & 0 \\[5pt]
    -\frac{ v_3 v_2}{(1-( v_3)^2)^2} & \frac{ v_3x'}{(1-( v_3)^2)^2} & 0
    \end{bmatrix}\cdot\nabla_{v} + \mbox{2nd order operator} \\
    &= 0 + \mbox{2nd order operator},
\end{align*}
where we write the second order term implicitly, since it will eventually be cancelled. Similarly,
\begin{align*}
    e_5\,e_4 &= \begin{bmatrix}
    - v_2 &  v_1 & 0
    \end{bmatrix}\begin{bmatrix}
    -\frac{ v_3}{1-( v_3)^2} & 0 & 0 \\[5pt]
    0 & -\frac{ v_3}{1-( v_3)^2} & 0 \\[5pt]
    -\frac{ v_1}{(1-( v_3)^2)^2} & -\frac{ v_2}{(1-( v_3)^2)^2} & -\frac{ v_3}{(1-( v_3)^2)}
    \end{bmatrix}\cdot\nabla_{v} + \mbox{2nd order operator} \\
    &= -\frac{ v_3}{\sqrt{1-( v_3)^2}}\,e_5 + \mbox{2nd order operator}.
\end{align*}
We can deduce our claim by combining the above two calculations. 
\end{proof}

Our next step is to see how the $\tilde{X}_k$'s are represented in this frame. It is a standard linear algebra exercise, and we record the result below with details omitted. 

\begin{lemma}\label{represent}
For the $\tilde{X}_k$'s, we have the following representaion 
\begin{align*}
\tilde{X}_0 &= \cos{ x_3}\,e_1 + \sin{ x_3}\,e_2 + e_3 + \frac{( v_3)^2}{\sqrt{1-( v_3)^2}}( v_1\sin{ x_3}- v_2\cos{ x_3})\,e_4 \\
&\quad + \frac{ v_3}{\sqrt{1-( v_3)^2}}( v_2\sin{ x_3} +  v_1\cos{ x_3})\,e_5,  \\[3pt]
\tilde{X}_1 &= e_3, \\[3pt]
\tilde{X}_k &= B_k\,e_3 + ( v_1\partial_{ x_1}B_k +  v_2\partial_{ x_2}B_k)\,e_4, \qquad k = 2,\dotsc, 5. 
\end{align*}
Here, $B_2 = \sin{ x_1}$, $B_3 = \cos{ x_1}$, $B_4 = \sin{ x_2}$ and $B_5 = \cos{ x_2}$.
\end{lemma}

Our strategy is to represent a collection of commutators of the vector fields $\tilde{X}_k$'s in this moving frame, and then verify the collection satisfies the spanning condition.

\begin{lemma}\label{spanningset}
For $w$ in $\mathbb{ST}^3$ such that $ v_3\neq\pm 1$, the following vector fields are represented in $\{e_1,\dotsc, e_5\}$ as
\begin{align*}
\tilde{X}_1 &= \mathbf{a}_1 \\[3pt]
\tilde{X}_k &= \mathbf{a}_k \\[3pt]
[\tilde{X}_1, \tilde{X}_0] &= \mathbf{a}_6\\[3pt]
[\tilde{X}_1,[\tilde{X}_1, \tilde{X}_0]] &= \mathbf{a}_7 \\[3pt]
[\tilde{X}_k,[\tilde{X}_1, \tilde{X}_0]] &= \mathbf{a}_{6+k} \\[3pt]
\big[\tilde{X}_k,[\tilde{X}_1[,\tilde{X}_1, \tilde{X}_0]]\big]& = \mathbf{a}_{10+k},
\end{align*}
for $k=2,\dotsc, 5$, where
\begin{align*}
\mathbf{a}_1 &= \begin{bmatrix}
0 & 0 & 1 & 0 & 0
\end{bmatrix} \\[3pt]
\mathbf{a}_k &= \begin{bmatrix}
0 & 0 & B_k &  v_1\partial_{ x_1}B_k +  v_2\partial_{ x_2}B_k & 0
\end{bmatrix} \\[3pt]
\mathbf{a}_6 &= \begin{bmatrix}
-\sin{ x_3} & \cos{ x_3} & 0 & \frac{( v_3)^2}{\sqrt{1-( v_3)^2}}( v_1\cos{ x_3} +  v_2\sin{ x_3}) & -\frac{ v_3}{\sqrt{1-( v_3)^2}}(- v_2\cos{ x_3} +  v_1\sin{ x_3})
\end{bmatrix} \\[3pt]
\mathbf{a}_7 &= \begin{bmatrix}
-\cos{ x_3} & -\sin{ x_3} & 0 & -\frac{( v_3)^2}{\sqrt{1-( v_3)^2}}( v_1\sin{ x_3} -  v_2\cos{ x_3}) & -\frac{ v_3}{\sqrt{1-( v_3)^2}}( v_2\sin{ x_3} +  v_1\cos{ x_3})
\end{bmatrix} \\[3pt]
\mathbf{a}_{6+k} &= a_k\,\mathbf{a}_7 + \begin{bmatrix}
0 & 0 & b_k & c_k & ( v_1\partial_{ x_1}B_k +  v_2\partial_{ x_2}B_k)\,\frac{1}{1-( v_3)^2}\,( v_2\cos{ x_3} -  v_1\sin{ x_3})
\end{bmatrix} \\[3pt]
\mathbf{a}_{10+k} &= a_k'\,\mathbf{a}_6 + \begin{bmatrix}
0 & 0 & b_k' & c_k' & -( v_1\partial_{ x_1}B_k +  v_2\partial_{ x_2}B_k)\,\frac{1}{1-( v_3)^2}\,( v_1\cos{ x_3} +  v_2\sin{ x_3})
\end{bmatrix},
\end{align*}
and $a_k$, $b_k$, $c_k$, $a_k'$, $b_k'$, $c_k'$ are smooth functions. 
\end{lemma}

\begin{proof}
The representations of $\tilde{X}_1, \tilde{X}_2, \dotsc, \tilde{X}_5$ follow directly from Lemma~\ref{represent}. We start with $\mathbf{a}_6$, 

\begin{align*}
[\tilde{X}_1, \tilde{X}_0] &= [e_3, \cos{ x_3}\,e_1 + \sin{ x_3}\,e_2 + e_3 \\ &\quad + \frac{( v_3)^2}{\sqrt{1-( v_3)^2}}( v_1\sin{ x_3}- v_2\cos{ x_3})\,e_4  +  \frac{ v_3}{\sqrt{1-( v_3)^2}}( v_2\sin{ x_3} +  v_1\cos{ x_3})\,e_5], 
\end{align*}
which simplifies to 
\begin{align*}
[\tilde{X}_1, \tilde{X}_0] &= -\sin{ x_3}e_1 + \cos{ x_3}e_2 + \frac{( v_3)^2}{\sqrt{1-( v_3)^2}}( v_1\cos{ x_3}+ v_2\sin{ x_3})\,e_4  \\
& \quad + \frac{ v_3}{\sqrt{1-( v_3)^2}}( v_2\cos{ x_3} -  v_1\sin{ x_3})\,e_5;
\end{align*}
this gives $\mathbf{a}_6$. The calculation of $\mathbf{a}_7$ can be done similarly.

Now, we compute $\mathbf{a}_{6+k}$ for $k=2,\dotsc, 5$. Due to bilinearity of the Lie bracket, we have
\begin{equation*}
[\tilde{X}_k,[\tilde{X}_1, \tilde{X}_0]] = [B_ke_3,[\tilde{X}_1, \tilde{X}_0]] + [( v_1\partial_{ x_1}B_k +  v_2\partial_{ x_2}B_k)e_4,[\tilde{X}_1, \tilde{X}_0]].
\end{equation*}
We compute each term seperately; starting with  
\begin{multline*}
\big[B_ke_3,\,[\tilde{X}_1, \tilde{X}_0]\big] = (-B_k\cos{ x_3}\,e_1 + \partial_{ x_1}B_k\,\sin{ x_3}\,e_3) + (-B_k\,\sin{ x_3}\,e_1 - \partial_{ x_2}B_k\,\cos{ x_3}\,e_3) \\
+ B_k\frac{( v_3)^2}{\sqrt{1-( v_3)^2}}\,(- v_1\sin{ x_3} +  v_2\cos{ x_3})\,e_4 + B_k\frac{ v_3}{\sqrt{1-( v_3)^2}}\,(- v_2\sin{ x_3} -  v_1\cos{ x_3})\,e_5,
\end{multline*}
which reduces to 
\begin{equation*}
\big[B_ke_3,\,[\tilde{X}_1, \tilde{X}_0]\big] = B_k\big[\tilde{X}_1,\,[\tilde{X}_1, \tilde{X}_0]\big] + (\partial_xB_k\sin{ x_3} - \partial_yB_k\cos{ x_3})\,\tilde{X}_1.
\end{equation*}
Then we consider
\begin{multline*}
\big[( v_1\partial_{ x_1}B_k +  v_2\partial_{ x_2}B_k)\,e_4, [\tilde{X}_1, \tilde{X}_0]\big] \\
= \big[( v_1\partial_{ x_1}B_k +  v_2\partial_{ x_2}B_k)\,e_4, [\tilde{X}_1, \tilde{X}_0] - \frac{ v_3}{\sqrt{1-( v_3)^2}}\,( v_2\cos{ x_3} -  v_1\sin{ x_3})\,e_5\big] \\
    + \big[( v_1\partial_{ x_1}B_k +  v_2\partial_{ x_2}B_k)\,e_4,\, \frac{ v_3}{\sqrt{1-( v_3)^2}}\,( v_2\cos{ x_3} -  v_1\sin{ x_3})\,e_5\big].
\end{multline*}
Since
\begin{equation*}
[\tilde{X}_1, \tilde{X}_0] - \frac{ v_3}{\sqrt{1-( v_3)^2}}\,( v_2\cos{ x_3} -  v_1\sin{ x_3})\,e_5
\end{equation*}
has zero component in the direction $e_5$, one can see that
\begin{equation*}
\big[( v_1\partial_{ x_1}B_k +  v_2\partial_{ x_2}B_k)\,e_4,\,[\tilde{X}_1, \tilde{X}_0] - \frac{ v_3}{\sqrt{1-( v_3)^2}}\,( v_2\cos{ x_3} -  v_1\sin{ x_3})\,e_5\big] = \alpha\,e_4, 
\end{equation*}
where $\alpha$ is a smooth function. For the second term,
\begin{equation*}
\big[( v_1\partial_{ x_1}B_k +  v_2\partial_{ x_2}B_k)\,e_4,\,\frac{ v_3}{\sqrt{1-( v_3)^2}}\,( v_2\cos{ x_3}- v_1\sin{ x_3})\,e_5\big] = \alpha''\,e_4 + \alpha'''\,e_5 + \alpha''''\,[e_4,e_5],
\end{equation*}
where
\begin{align*}
\alpha''' &= ( v_1\partial_{ x_1}B_k +  v_2\partial_{ x_2}B_k)\,e_4\left( \frac{ v_3}{\sqrt{1-( v_3)^2}}( v_2\cos{ x_3- v_1\sin{ x_3}}) \right) 
\\[3pt]
\alpha'''' &= ( v_1\partial_{ x_1}B_k +  v_2\partial_{ x_2}B_k)\frac{ v_3}{\sqrt{1-( v_3)^2}}\,( v_2\cos{ x_3} -  v_1\sin{ x_3}),
\end{align*}
and $e_4$ denotes the directional derivative. One can show
\begin{equation*}
\alpha'''e_5 = ( v_1\partial_{ x_1}B_k +  v_2\partial_{ x_2}B_k)( v_2\cos{ x_3} -  v_1\sin{ x_3})\,e_5,
\end{equation*}
and
\begin{align*}
\alpha''''\,[e_4,e_5] &= ( v_1\partial_{ x_1}B_k +  v_2\partial_{ x_2}B_k)\frac{ v_3}{\sqrt{1-( v_3)^2}}\,( v_2\cos{ x_3} -  v_1\sin{ x_3})\frac{ v_3}{\sqrt{1-( v_3)^2}}\,e_5 \\[3pt]
&= ( v_1\partial_{ x_1}B_k +  v_2\partial_{ x_2}B_k)( v_2\cos{ x_3} -  v_1\sin{ x_3})\frac{( v_3)^2}{1-( v_3)^2}\,e_5,
\end{align*}
which we can substitute back into the above to get
\begin{multline*}
\big[( v_1\partial_{ x_1}B_k +  v_2\partial_{ x_2}B_k)\,e_4,\,\frac{ v_3}{\sqrt{1-( v_3)^2}}\,( v_2\cos{ x_3}- v_1\sin{ x_3})\,e_5\big] = \alpha''\,e_4 \\ 
+ ( v_1\partial_{ x_1}B_k +  v_2\partial_{ x_2}B_k)( v_2\cos{ x_3} -  v_1\sin{ x_3})\frac{1}{1-( v_3)^2}\,e_5.
\end{multline*}
Therefore, we have
\begin{multline*}
[\tilde{X}_k,[\tilde{X}_1, \tilde{X}_0]] = B_k\big[\tilde{X}_1,\,[\tilde{X}_1, \tilde{X}_0]\big] + (\partial_xB_k\sin{ x_3} - \partial_yB_k\cos{ x_3})\,\tilde{X}_1 + \alpha''\,e_4 \\ 
+ ( v_1\partial_{ x_1}B_k +  v_2\partial_{ x_2}B_k)( v_2\cos{ x_3} -  v_1\sin{ x_3})\frac{1}{1-( v_3)^2}\,e_5,
\end{multline*}
which shows our claim for $\mathbf{a}_{6+k}$ where $k=2,\dotsc, 5$. The calculation for $\mathbf{a}_{10+k}$ where $k=2,\dotsc, 5$ can be done similarly. 
\end{proof}

Finally, we are in the position to verify Theorem~\ref{hypoellipticitythm1}; we show the collection in the previous lemma gives a spanning set of the tangent space in $\{ v_3\neq\pm 1\}$.

\begin{proof}
Our goal is to show that $\{\mathbf{a}_1, \mathbf{a}_2, \dotsc, \mathbf{a}_{15}\}$ spans $\mathbb{R}^5$ when $v_3\neq\pm 1$. First of all, we observe that
\begin{equation*}
\mathbf{a}_{k,4} =  v_1\partial_{ x_1}B_k +  v_2\partial_{ x_2}B_k \quad\mbox{for}\quad k=2,\dotsc, 5
\end{equation*}
cannot all vanish at the same time. Indeed, 
\begin{equation*}
\mathbf{a}_{2,4} =  v_1\cos{ x_1}, \quad \mathbf{a}_{3,4} = - v_1\sin{ x_1}, \quad \mathbf{a}_{4,4} =  v_2\cos{ x_2} \quad\mbox{and}\quad \mathbf{a}_{5,4} = - v_2\sin{ x_2},
\end{equation*}
and we know $ v_1$ and $ v_2$ cannot vanish at the same time, since $( v_1)^2 + ( v_2)^2 = 1 - ( v_3)^2 > 0$; hence, we may pick some $k_1\in\{2,\dotsc, 5\}$ such that $\mathbf{a}_{k_1, 4}\neq 0$. 

Next, denote
\begin{equation*}
\mathbf{b}_{6+k} = \mathbf{a}_{6+k} - a_k\,\mathbf{a}_7 \quad\mbox{and}\quad \mathbf{b}_{10+k} = \mathbf{a}_{10+k} - a_k\,\mathbf{a}_6
\end{equation*}
for $k=2,\dotsc, 5$; our goal is to show that there exists $\mathbf{b}_{k_2}$ in $\{\mathbf{b}_8,\mathbf{b}_9,\dotsc, \mathbf{b}_{15}\}$ such that $\mathbf{b}_{k_2, 5}$ does not vanish. This reduces to showing that
\begin{equation*}
 v_2\cos{ x_3} -  v_1\sin{ x_3} \quad\mbox{and}\quad  v_1\cos{ x_3} +  v_2\sin{ x_3}
\end{equation*}
cannot vanish at the same time. By contradiction, we suppose otherwise, and we may without loss of generality assume $\cos{ x_3}$ and $ v_1$ are not zero. Then
\begin{equation*}
 v_2\cos{ x_3} -  v_1\sin{ x_3} =  v_1\cos{ x_3} +  v_2\sin{ x_3} = 0 \quad\Rightarrow\quad  \frac{ v_2}{ v_1} = \tan{ x_3} = -\frac{ v_1}{ v_2}.
\end{equation*}
Also, $ v_1\cos{ x_3} +  v_2\sin{ x_3} = 0$ implies that $\sin{ x_3}$ and $ v_2$ are also non-vanishing under our hypothesis. We then deduce
\begin{equation*}
\frac{ v_2}{ v_1} = -\frac{ v_1}{ v_2} \quad\Rightarrow\quad -( v_1)^2 = ( v_2)^2,
\end{equation*}
which is possible if and only if $ v_1= v_2=0$; this is a contradiction, since $( v_1)^2 + ( v_2)^2 = 1 - ( v_3)^2 > 0$. Hence, we may pick $k_2 = 6+k_1$ if $ v_2\cos{ x_3} -  v_1\sin{ x_3}\neq 0$ or $k_2 = 10+k_1$ if $ v_1\cos{ x_3} +  v_2\sin{ x_3}\neq 0$. 

Finally, we conclude that $\{\mathbf{a}_1, \mathbf{a}_{k_1}, \mathbf{a}_6, \mathbf{a}_7, \mathbf{a}_{k_2}\}$ spans $\mathbb{R}^5$. Indeed, by putting these vectors together, we obtain a matrix
\begin{equation*}
\begin{bmatrix}
--- & \mathbf{a}_1 & --- \\[3pt]
--- & \mathbf{a}_{k_1} & --- \\[3pt]
--- & \mathbf{a}_6 & --- \\[3pt]
--- & \mathbf{a}_7 & --- \\[3pt]
--- & \mathbf{a}_{k_2} & ---
\end{bmatrix},
\end{equation*}
which is row equivalent to 
\begin{equation*}
\begin{bmatrix}
0 & 0 & 1 & 0 & 0 \\[3pt]
0 & 0 & * & \alpha & 0 \\[3pt]
-\sin{ x_3} & \phantom{-}\cos{ x_3} & 0 & * & * \\[3pt]
-\cos{ x_3} & -\sin{ x_3} & 0 & * & * \\[3pt]
0 & 0 & * & * &  \beta
\end{bmatrix},
\end{equation*}
where $*$ are the unimportant entries, and $\alpha$, $\beta$ are non-zero. It is clear that the matrix is row equivalent to the block matrix,
\begin{equation*}
\begin{bmatrix}
-\sin{ x_3} & \phantom{-}\cos{ x_3} & 0 & 0 & 0 \\[3pt]
-\cos{ x_3} & -\sin{ x_3} & 0 & 0 & 0 \\[3pt]
0 & 0 & 1 & 0 & 0 \\[3pt]
0 & 0 & 0 & \alpha & 0 \\[3pt]
0 & 0 & 0 & 0 & \beta
\end{bmatrix},
\end{equation*}
where the first block is an orthogonal matrix, and the second block is of full rank. So the entire matrix is of full rank, and the proof is completed. 
\end{proof}

\subsection{Irreducibility}\label{irreducibility}

In this section, we verify the irreducibility for the projective process $w_t$. In the following, we will use subscripts such as $x$ and $v$ to denote the associated components of a given position or vector field. 

Recall that for a given family of vector fields $\mathcal{F}$, a continuous curve $\eta: [0,t]\rightarrow N$ is an integral curve of $\mathcal{F}$ if there is a partition $0= t_0 < t_1 <\dotsc<t_m = t$ and vector fields $V_1,\dotsc, V_m$ in $\mathcal{F}$ such that
\begin{equation*}
\dot{\eta}(t) = V_i(\eta(t)) \quad\mbox{for}\quad t_{i-1} < t < t_i, \quad i=1,\dotsc, m.
\end{equation*}
In the following, we will construct our control scheme iteratively: hence, it is convenient to define the extension of an integral curve. We say an integral curve $\tilde{\eta}: [0,t']\rightarrow N$ is an extension of $\eta: [0,t]\rightarrow N$, if
\begin{enumerate}
\item $[0,t]\subseteq [0,t']$
\item $\left.\tilde{\eta}\right\vert_{[0,t]} = \eta$
\item $\left.\tilde{\eta}\right\vert_{[t,t']}$ is an integral curve of $\mathcal{F}$.
\end{enumerate}
To simplify our notation, we often just use the same notation for the extension and the original function, when it is clear from the context. 

Let us start with the control problem for the base process $x_t$. An important point is that we only need $X_0$ and $X_1$ to get the irreducibility of the base process. 

\begin{theorem}\label{irreducibilitythm1}
Let $\mathcal{F}'= X_0 + \operatorname{span}{X_1}$. There exists $T_1>0$ such that 
\begin{equation}\label{irreducibilityeq6}
\overline{\mathcal{A}_{\mathcal{F}'}(x', T_1)} = \mathbb{T}^3 \quad\forall x'\in\mathbb{T}^3.
\end{equation}
Moreover, given an open neighborhood $V$ of $\left(x_1', x_2'\right)$ in $\mathbb{T}^2$ and $t\geq 0$,
\begin{equation}\label{irreducibilityeq5}
\overline{\left\{(\phi_t^1,\,\phi_t^2)\in V: \mbox{$\phi_t = \left(\phi_t^1,\,\phi_t^2,\,\phi_t^3\right)$ is an integral curve of $\mathcal{F}'$ with initial $x'$}\right\}} \neq\emptyset,
\end{equation}
i.e., we can fix the position in $(x_1,x_2)$ around its initial. 
\end{theorem}

\begin{proof}
Given a initial $x'$ and $c\neq 0$, a straightforward calculation shows the integral curve of  
\begin{equation*}
\phi_t = x' + \int_0^t X_0(\phi_s)\,ds + \int_0^t(-1+c)X_1(\phi_s)\,ds
.\end{equation*}
is
\begin{align*}
\phi_t^1 &= x_1'-\frac{1}{c}\sin{(x_3')} + \frac{1}{c}\sin{(x_3'+ct)} \\
\phi_t^2 &= x_2' + \frac{1}{c}\cos{(x_3')} - \frac{1}{c}\cos{(x_3'+ct)} \\[3pt]
\phi_t^3 &= x_3' + ct.
\end{align*}
One can observe that the trajectory of $\left(\phi_t^1,\, \phi_t^2\right)$ forms a circle passing through $(x_1', x_2')$ with radius $1/c$ and center $(x_1'-\sin(x_3')/c,\, x_2' + \cos(x_3')/c)$. By adjusting $c$, we may join any two points in the $x_1$-$x_2$ plane by such circle. In particular, there exists $T_1'$ such that
\begin{equation*}
\overline{\left\{\left(\phi_t^1, \phi_t^2\right)\in\mathbb{T}^2: \mbox{$\phi_t = \left(\phi_t^1, \phi_t^2, \phi_t^3\right)$ is an integral curve of $\mathcal{F}'$ with initial $x'$, $t\leq T_1'$}\right\}} = \mathbb{T}^2.
\end{equation*}
Note also that by picking $c\gg 1$ we can trap $\left(\phi_t^1,\,\phi_t^2\right)$ in any open neighborhood of its initial position $\left(x_1',\,x_2'\right)$, so  (\ref{irreducibilityeq5}) is proved. 

Finally, we deduce the controllability of $x_t$. Let $T_1 > T_1'$. Fix $U$ a target open set, $\tilde{U}$ a open subset of $U$, and $\epsilon>0$ such that
\begin{equation*}
\mathcal{A}_{X_0}(\tilde{U}, \epsilon) \subset U. 
\end{equation*}
We may without loss of generality assume $U = I_1\times I_2\times I_3$, and $\tilde{U} = \tilde{I}_1\times \tilde{I}_2\times \tilde{I}_3$, where $I_k$ and $\tilde{I}_k$ are open sets in $\mathbb{T}$ for $k=1, 2, 3$. We follow the steps below to obtain (\ref{irreducibilityeq6}); some pictures explaining the strategy are in Figure~\ref{irreducibilitypic1}.
\begin{enumerate}
\item Construct an integral curve $\psi: [0,t]\rightarrow \mathbb{T}^3$ of $\mathcal{F}'$ such that $\left(\psi_t^1, \psi_t^2 \right)\in \tilde{I}_1\times \tilde{I}_2$. 
\item Extend $\psi$ to an integral curve $\psi: [0,T_1-\epsilon]\rightarrow \mathbb{T}^3$ of $\mathcal{F}'$ such that $\left(\psi_{T_1-\epsilon}^1, \psi_{T_1-\epsilon}^2\right)\in \tilde{I}_1\times \tilde{I}_2$.
\item Extend $\psi$ to an integral curve $\psi: [0, T_1-\epsilon + \gamma^{-1}]\rightarrow\mathbb{T}^3$ such that 
\begin{equation*}
\psi_{T_1-\epsilon + \gamma^{-1}} = \psi_{T_1-\epsilon} + \int_0^{1/\gamma}X_0(\psi_s)\,ds + \int_0^{1/\gamma}\gamma c'X_1(\psi_s)\,ds
\end{equation*}
for some $\gamma\gg 1$ and suitable $c'$, so that $\psi_{T_1-\epsilon + \gamma^{-1}}\in\tilde{U}$.
\item Finally, extend $\psi$ to an integral curve $\psi: [0,T_1]\rightarrow\mathbb{T}^3$ of $\mathcal{F}'$ such that
\begin{equation*}
\psi_{T_1} = \psi_{T_1-\epsilon+1/\gamma} + \int_0^{\epsilon-1/\gamma} X_0(\psi_s)\,ds.
\end{equation*}
Due to our choice of $\epsilon$, and the fact that $\psi_{T_1-\epsilon+1/\gamma}\in \tilde{U}$, we have $\psi_{T_1}\in U$. 
\end{enumerate}

This concludes the proof of Theorem~\ref{irreducibilitythm1}.
\end{proof}

\begin{figure}
\begin{subfigure}{.5\textwidth}
  \centering
  \includegraphics[width=.8\linewidth]{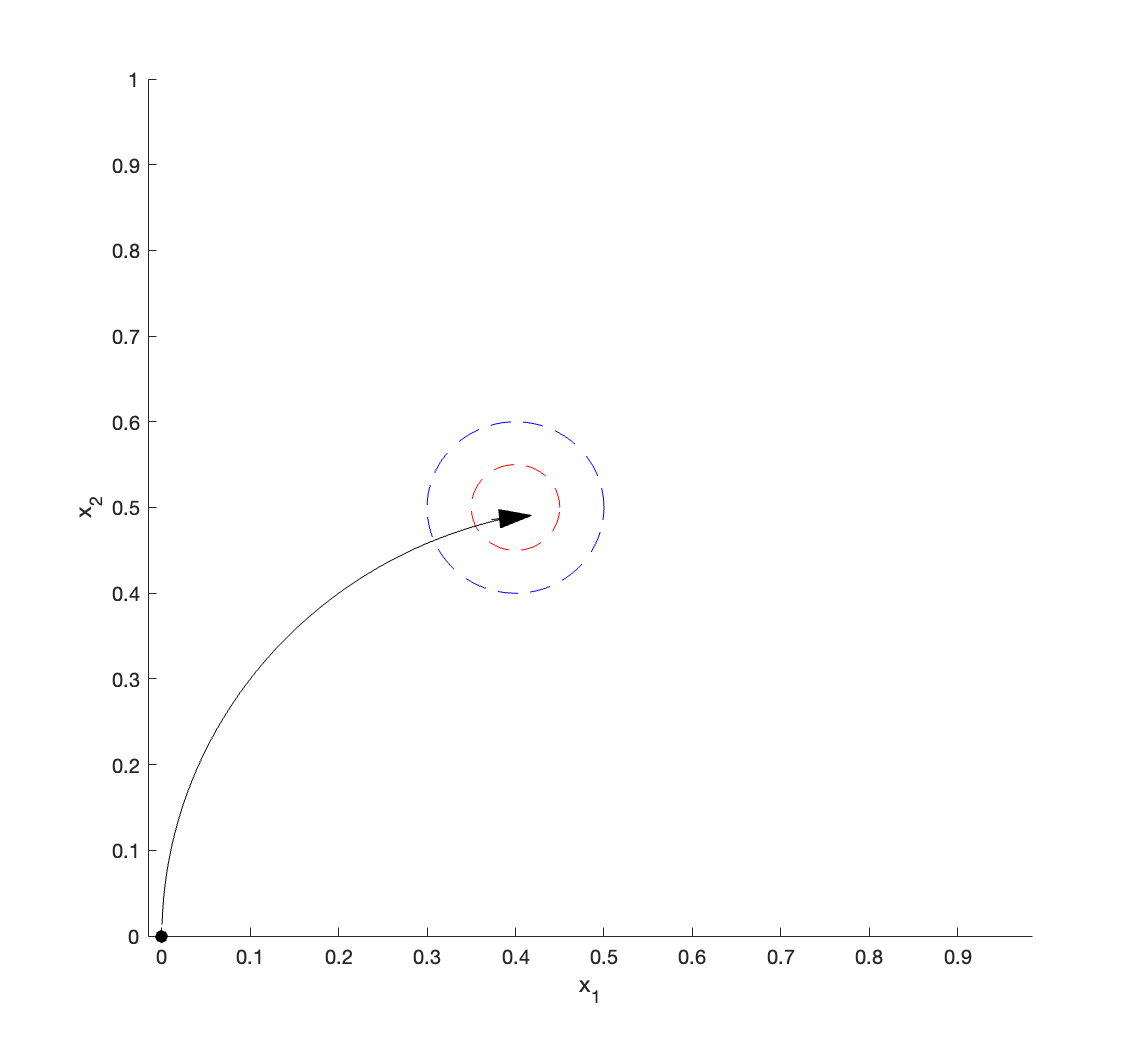}
  \caption{Step~1: $(x_1,\,x_2)$ move to the red neighborhood}
  \label{fig:sfig1}
\end{subfigure}%
\begin{subfigure}{.5\textwidth}
  \centering
  \includegraphics[width=.8\linewidth]{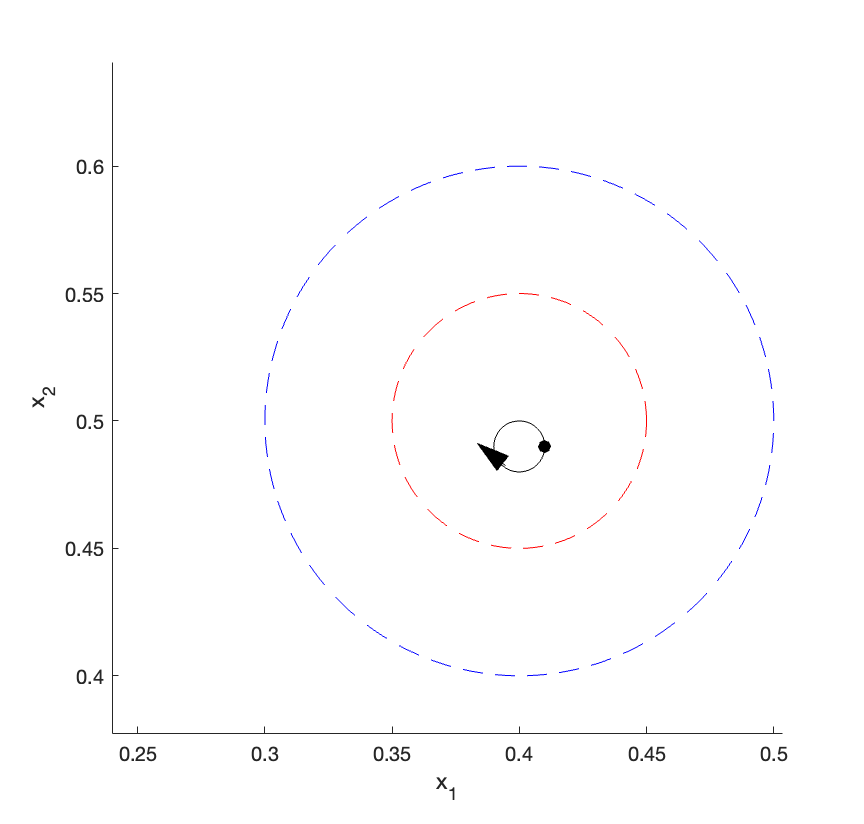}
  \caption{Step~2: $(x_1,\,x_2)$ idle inside the red neighborhood.}
  \label{fig:sfig2}
\end{subfigure}
\begin{subfigure}{.5\textwidth}
  \centering
  \includegraphics[width=.8\linewidth]{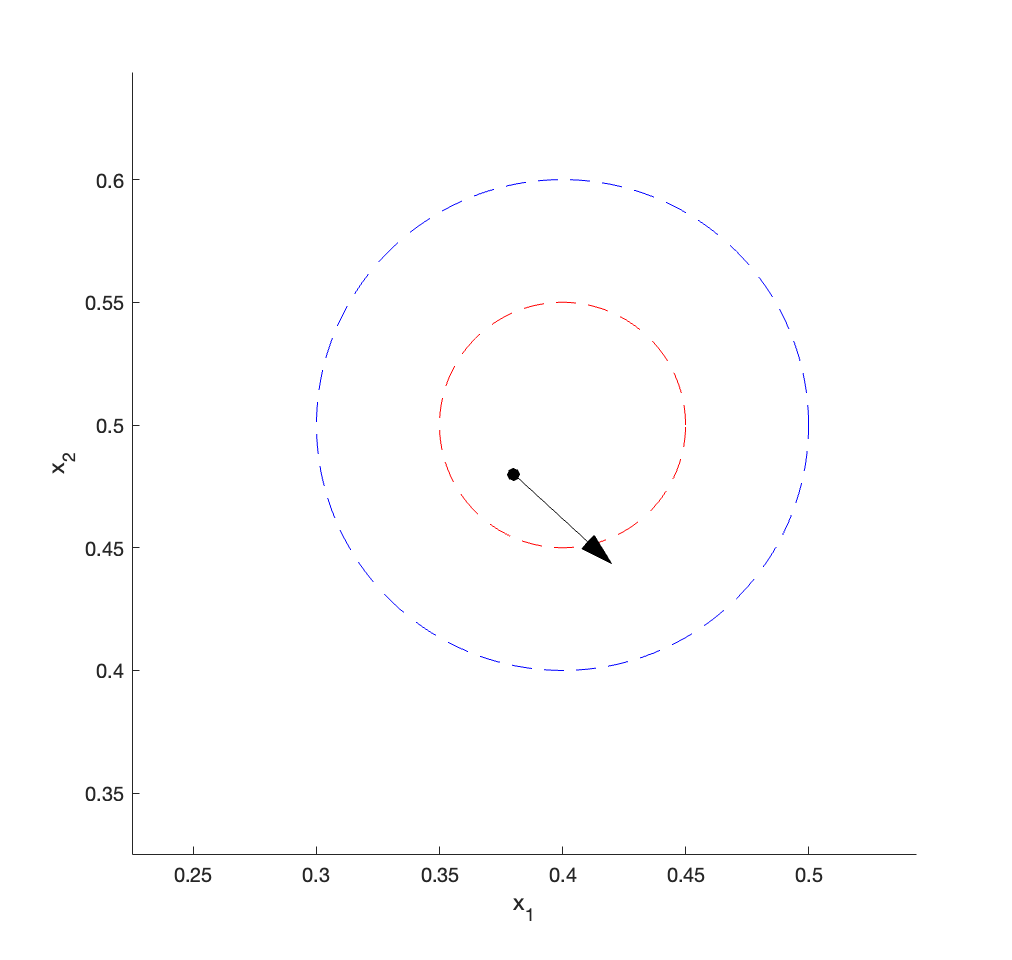}
  \caption{Step~3: $(x_1,\,x_2)$ drift inside the blue neighborhood.}
  \label{fig:sfig3_1}
\end{subfigure}
\begin{subfigure}{.5\textwidth}
  \centering
  \includegraphics[width=.8\linewidth]{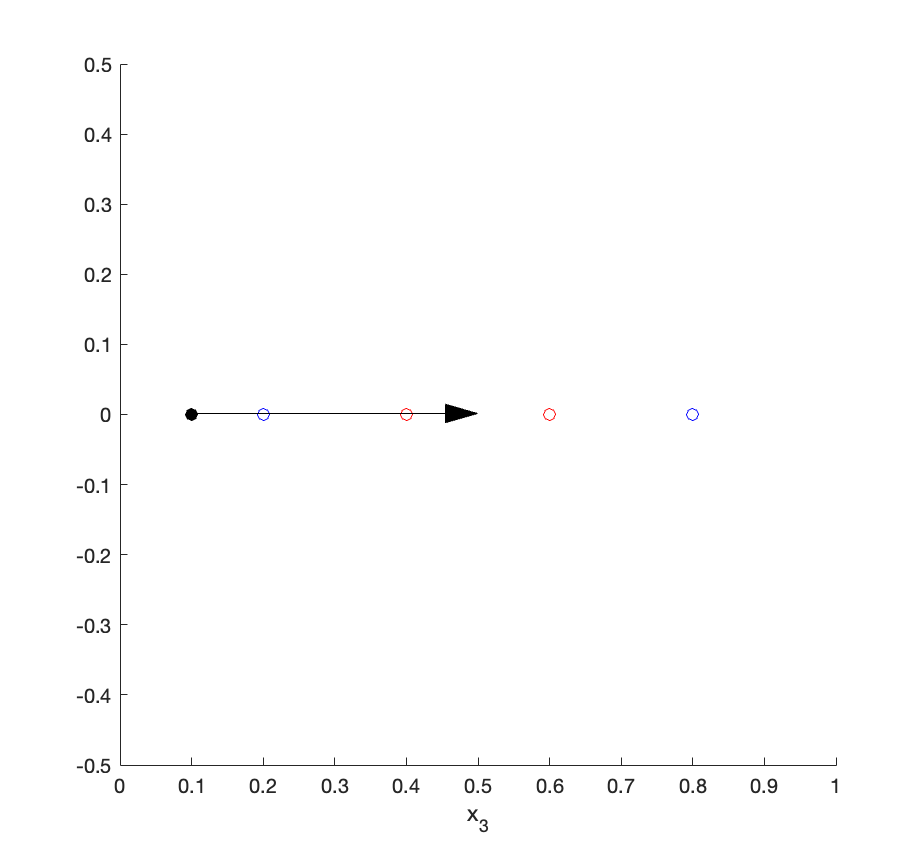}
  \caption{Step~3: $x_3$ move to the red neighborhood.}
  \label{fig:sfig3_2}
\end{subfigure}
\begin{subfigure}{.5\textwidth}
  \centering
  \includegraphics[width=.8\linewidth]{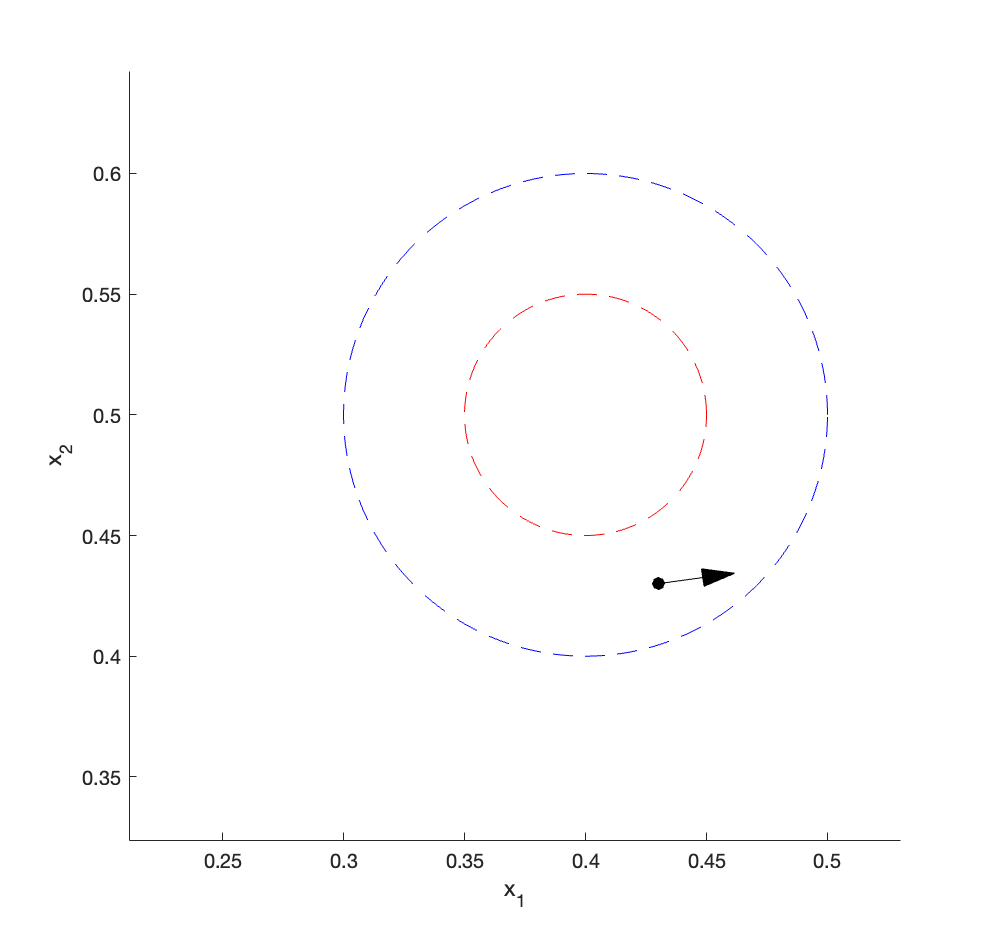}
  \caption{Step~4: $(x_1,\,x_2)$ drift inside the blue neighborhood.}
  \label{fig:sfig4_1}
\end{subfigure}
\begin{subfigure}{.5\textwidth}
  \centering
  \includegraphics[width=.8\linewidth]{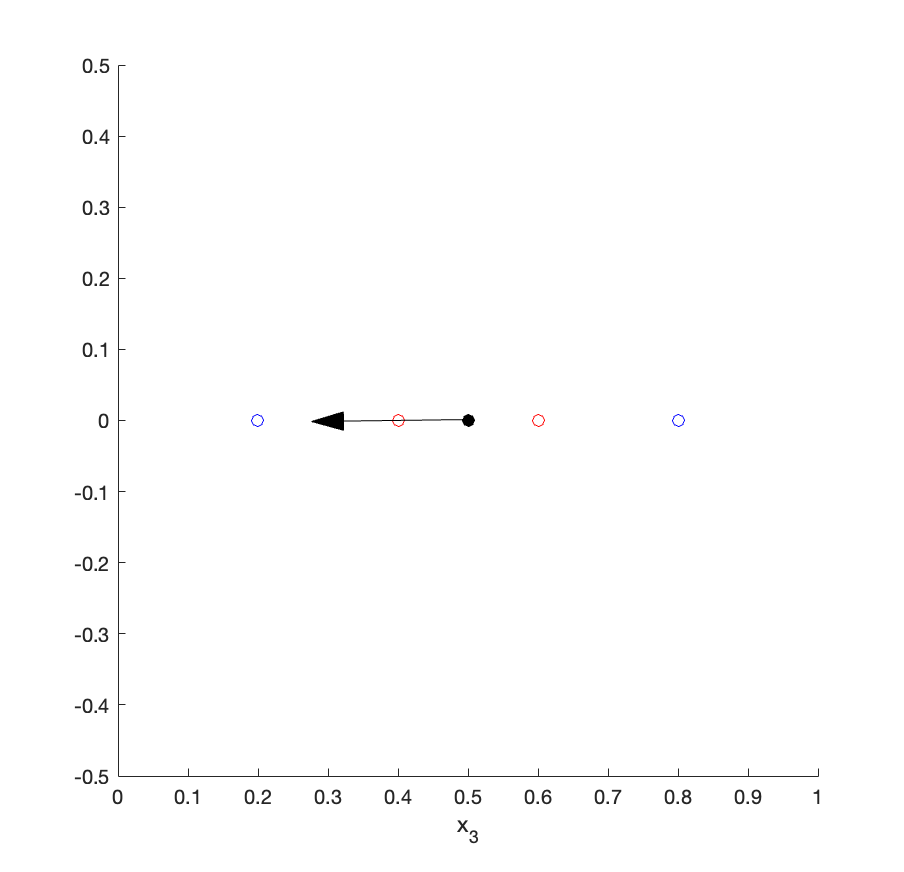}
  \caption{Step~4: $x_3$ drift inside the blue neighborhood.}
  \label{fig:sfig4_2}
\end{subfigure}
\caption{In these pictures, the blue region represents the target neighborhood, the red region represents the smaller neighborhood inside, and the black line represents the trajectory.}\label{irreducibilitypic1}
\end{figure}

To control the projective process, our strategy is to split the control problem into two parts, one on the unit sphere, and the other one on the base manifold. To coordinate the two control problems, we use the topological property of the even-dimensional sphere, hairy ball theorem. First, we move the component on the sphere to the points where $\nabla_xX_0v$ vanishes. Second, we use $X_0$ and $X_1$ to achieve the control in the $(x_1,x_2)$ component of the base manifold. Finally, we finish the control problem on the sphere and the $x_3$ component of the base manifold to close the argument. 

Let us make a few observations. Before projecting onto the tangent plane of $S_x\mathbb{T}^3$, the $v$-component of $\tilde{X}_0$ is
\begin{equation*}
\nabla_x X_0\,v = \begin{bmatrix}
-v_3\sin{x_3} & v_3\cos{x_3} & 0
\end{bmatrix}.
\end{equation*}
Using the diffusion vector field $\tilde{X}_1$, we can move instantaneously along the $x_3$-direction; therefore it makes sense to consider the following auxilliary vector fields by letting $x_3$ be $0$, $\pi/2$, $\pi$ and $3\pi/2$,  
\begin{align*}
Y_{0,1} &\coloneqq \left(1-\Pi_v\right)\nabla_x X_0(x_1,x_2, 0 )v = \left(1-\Pi_v\right)\begin{bmatrix}
\phantom{-}0 & \phantom{-}v_3 & 0
\end{bmatrix}, \\[3pt]
Y_{0,2} &\coloneqq \left(1-\Pi_v\right)\nabla_x X_0(x_1,x_2,\pi/2)v = \left(1-\Pi_v\right)\begin{bmatrix}
-v_3 & \phantom{-}0 & 0
\end{bmatrix}, \\[3pt]
Y_{0,3} &\coloneqq \left(1-\Pi_v\right)\nabla_x X_0(x_1,x_2, \pi)v = \left(1-\Pi_v\right)\begin{bmatrix}
\phantom{-}0 & -v_3 & 0
\end{bmatrix}, \\[3pt]
Y_{0,4} &\coloneqq \left(1-\Pi_v\right)\nabla_x X_0(x_1,x_2, 3\pi/2)v = \left(1-\Pi_v\right)\begin{bmatrix}
\phantom{-}v_3 & \phantom{-}0 & 0
\end{bmatrix}.
\end{align*}
One can see that these vector fields vanish when $v_3=0$, and they allow us to move in the directions transversal to $v_3$-direction when $v_3\neq 0$. Next, we obtain from the $v$-component of the diffusion vector fields,  
\begin{align*}
Y_1 &= v_1\,\begin{bmatrix}
-v_3v_1 & -v_3v_2 & 1-(v_3)^2
\end{bmatrix}, \\[3pt]
Y_2 &= v_2\,\begin{bmatrix}
-v_3v_1 & -v_3v_2 & 1-(v_3)^2
\end{bmatrix}.
\end{align*}
Define the auxiliary vector fields $\mathcal{G} = \cup_{k=1}^4Y_{0,k} \,\cup\, \mbox{span}\left\{Y_1, Y_2\right\}$ on $\mathbb{S}^2$. An important observation is 
\begin{theorem}\label{irreducibilitythm2}
There exists $T>0$ such that
\begin{equation*}
\{v_3 =0\} \subseteq \mathcal{A}_{\mathcal{G}}(v, \leq T) \quad\forall v\in\mathbb{S}^2.
\end{equation*} 
Moreover, $\forall\epsilon>0$
\begin{equation*}
\overline{\mathcal{A}_{\mathcal{G}}(\{v_3 =0\}, \leq \epsilon)} = \mathbb{S}^2.
\end{equation*}
\end{theorem}

The control problem of the auxiliary family $\mathcal{G}$ is closely related to projective irreducibility. First, we show that we can approximate the integral curves of $\mathcal{G}$ arbitrarily well with the $v$-component of vector fields in $\mathcal{F} = \tilde{X}_0 + \tilde{\mathcal{X}}$. 

\begin{theorem}\label{irreducibilitythm3}
Let $w' = (x',v')$, $ w'' = (x'',v'') \in \mathbb{S}\mathbb{T}^3$. Let 
\begin{align*}
\varphi_t &= v' + \int_0^{t}Y(\varphi_s)\,ds \quad\mbox{for}\quad Y\in\{Y_{0,1},\,Y_{0,2},\,Y_{0,3},\,Y_{0,4}\} \\
\phi_{t/\gamma} &= v' + \int_0^{t/\gamma}\gamma Y(\phi_{s})\,ds \quad\mbox{for}\quad Y\in\mbox{span}\{Y_1, Y_2\}, \quad \gamma\in\mathbb{R}.
\end{align*}
Denote $M$ and $L$ constants depending on the sup-norm and the Lipschitz constant of the vector fields. Then for all $0 < \delta \ll 1$, there exist controls $\alpha$ and $\beta$ such that $\psi_{t+\delta}^{\alpha}(w'')$ and $\psi_{t/\gamma}^{\beta}(w'')$ are the integral curves of $\mathcal{F}$ with initial $w''$ under the controls $\alpha$ and $\beta$ respectively, and  
\begin{align}
\left\vert \left(\psi_{t+\delta}^{\alpha, 4},\,\psi_{t+\delta}^{\alpha, 5},\,\psi_{t+\delta}^{\alpha, 6}\right) - \varphi_t \right\vert &\leq \left(\left\vert v'' - v'\right\vert + \delta M \right) e^{Lt} \label{irreducibilityeq7} \\[3pt]
\left\vert \left(\psi_{t/\gamma}^{\beta,4},\,\psi_{t/\gamma}^{\beta, 5},\,\psi_{t/\gamma}^{\beta, 6}\right) - \phi_{t/\gamma}\right\vert 
&\leq \left( \frac{Mt}{\gamma}  + \vert v' - v''\vert\right) e^{2Lt}. \label{irreducibilityeq8}
\end{align}
\end{theorem}

\begin{proof}
If $Y\in\{Y_{0,1},\,Y_{0,2},\,Y_{0,3},\,Y_{0,4}\}$, we may consider $Y=Y_{0,1}$ without loss of generality. We pick the control $\alpha = (\alpha^1,\dotsc, \alpha^5): [0, \infty)\rightarrow \mathbb{R}^5$,
\begin{equation*}
\alpha(s) = -(1+\frac{x_3''}{\delta})\mathbf{e}_1\chi_{[0,\delta)}(s) - \mathbf{e}_1\chi_{[\delta,\infty)}(s)
\end{equation*}
where $\mathbf{e}_1\in\mathbb{R}^5$ is the unit vector in the first entry; this control moves the $x_3$ component from $x_3''$ to $0$ within time $\delta$, and then fixes the $x_3$ component after time $\delta$. Let
\begin{align*}
\psi_{t+\delta}^{\alpha}(w'') &= w'' + \int_0^{t+\delta}  \tilde{X}_0(\psi_s^{\alpha}(w''))\,ds + \int_0^{t+\delta}\tilde{X}_1(\psi_s^{\alpha}(w''))\,\alpha^1(s)ds \\[3pt]
\tilde{\psi}_t &= \psi_{\delta}^{\alpha}(w'') + \int_0^{t}\tilde{X}_0(\tilde{\psi}_s)\,ds - \int_0^t\tilde{X}_1(\tilde{\psi}_s)\,ds \\[3pt]
\varphi_t(v') &= v' + \int_0^{t}Y_{0,1}(\varphi_s(v'))\,ds,
\end{align*}
where $\psi_{\cdot}^{\alpha}(w'')$ is the trajectory under the control $\alpha$ with initial $w''$, $\tilde{\psi}_{\cdot}$ is an auxiliary trajectory for estimation, and $\varphi_{\cdot}(v')$ is the trajectory we aim to approximate. We have
\begin{equation*}
\left(\tilde{\psi}_t^4,\,\tilde{\psi}_t^5,\,\tilde{\psi}_t^6\right)  - \varphi_t = \left(\psi_{\delta}^{\alpha,4}(w''),\,\psi_{\delta}^{\alpha,6}(w''),\,\psi_{\delta}^{\alpha,6}(w'')\right) - v' + \int_0^{t}Y_{0,1}(\tilde{\psi}_s)\,ds -  \int_0^{t}Y_{0,1}(\varphi_s)\,ds,
\end{equation*}
which by Gr\"{o}nwall inequality implies
\begin{align}
\left\vert\left(\tilde{\psi}_t^4,\,\tilde{\psi}_t^5,\,\tilde{\psi}_t^6\right)  - \varphi_t \right\vert &\leq \left\vert\left(\psi_{\delta}^{\alpha,4}(w''),\,\psi_{\delta}^{\alpha,5}(w''),\,\psi_{\delta}^{\alpha,6}(w'')\right) - v'\right\vert e^{Lt} \\[3pt]
&\leq \left(\left\vert v'' - v'\right\vert + \left\vert \int_0^{\delta}  \tilde{X}_0(\psi_s^{\alpha}(w''))\,ds\right\vert \right) e^{Lt}, \label{irreducibilitythm3eq1}
\end{align}
where $L$ depends on the Lipschitz constant of the vector fields, and the integral involving $\tilde{X}_1$ vanishes due to $(\tilde{X}_1^4,\tilde{X}_1^5,\tilde{X}_1^6) = 0$. Also, we have
\begin{equation}\label{irreducibilitythm3eq2}
\psi_{t+\delta}^{\alpha}(w'') = \tilde{\psi}_t
\end{equation}
by construction. Combining (\ref{irreducibilitythm3eq1}) and (\ref{irreducibilitythm3eq2}),
\begin{align*}
\left\vert \left(\psi_{t+\delta}^{\alpha, 4},\,\psi_{t+\delta}^{\alpha, 5},\,\psi_{t+\delta}^{\alpha, 6}\right) - \varphi_t \right\vert &= \left\vert\left(\tilde{\psi}_t^4,\,\tilde{\psi}_t^5,\,\tilde{\psi}_t^6\right)  - \varphi_t \right\vert \\[3pt]
& \leq \left(\left\vert v'' - v'\right\vert + \left\vert \int_0^{\delta}  \tilde{X}_0(\psi_s^{\alpha}(w''))\,ds\right\vert \right) e^{Lt} \leq \left(\left\vert v'' - v'\right\vert + \delta M \right) e^{Lt},
\end{align*}
where $M$ depends on the sup norm of the vector fields; this shows (\ref{irreducibilityeq7}).

If $Y\in\{Y_1,Y_2\}$, we may consider $Y= Y_1$ without loss of generality. At least one of $\cos^2{x_1''}$ and $\sin^2{x_1''}$ is greater than or equal to $1/2$. For definiteness, we assume $\left\vert\cos{x_1''}\right\vert \geq \cos^2({x_1''}) \neq 1/2$. Given $\gamma >0$, we pick the  control $\beta = (\beta^1,\dotsc, \beta^5):[0,\infty)\rightarrow \mathbb{R}^5$, 
\[
\beta(s) = \frac{\gamma}{\cos{x_1''}}\mathbf{e_2},
\]
where $\mathbf{e}_2\in\mathbb{R}^5$ is the unit vector in the second entry. We then consider the following
\begin{align*}
\psi_{t/\gamma}^{\beta} &= w'' + \int_0^{t/\gamma}  \tilde{X}_0(\psi_s^{\beta})\,ds + \int_0^{t/\gamma}\frac{\gamma}{\cos{x_1''}}\tilde{X}_2(\psi_s^{\beta})\,ds \\[3pt]
\tilde{\psi}_{t/\gamma} &= w'' + \int_0^{t/\gamma}\frac{\gamma}{\cos{x_1''}}\tilde{X}_2(\tilde{\psi}_s)\,ds \\[3pt]
\phi_{t/\gamma} &= v' + \int_0^{t/\gamma}\gamma Y_1(\phi_{s})\,ds,
\end{align*}
where $\psi_{\cdot/\gamma}^{\beta}$ is the trajectory under the control $\beta$ with initial $w''$, $\tilde{\psi}_{\cdot/\gamma}$ is an auxiliary trajectory for estimation,  and $\phi_{\cdot/\gamma}$ is the trajectory we aim to approximate. 
\begin{equation*}
\psi_{t/\gamma}^{\beta} - \tilde{\psi}_{t/\gamma} =  \int_0^{t/\gamma}\tilde{X}_0(\psi_s^{\beta})\,ds + \int_0^{t/\gamma}\frac{\gamma}{\cos{x_1''}}\tilde{X}_2(\psi_s^{\beta})\,ds - \int_0^{t/\gamma}\frac{\gamma}{\cos{x_1''}}\tilde{X}_2(\tilde{\psi}_s)\,ds, 
\end{equation*}
which we may apply Gr\"{o}nwall inequality to get
\begin{equation}\label{irreducibilitythm3eq3}
\left\vert \psi_{t/\gamma}^{\beta} - \tilde{\psi}_{t/\gamma}\right\vert \leq  \frac{Mt}{\gamma} e^{Lt/\vert\cos{x_1''}\vert}.
\end{equation}
Also, 
\begin{align*}
\left(\tilde{\psi}_{t/\gamma,v}^4,\,\tilde{\psi}_{t/\gamma,v}^5,\,\tilde{\psi}_{t/\gamma,v}^6\right) &- \phi_{t/\gamma} = v'' - v' +  \int_0^{t/\gamma}\frac{\gamma}{\cos{x_1''}}\left(\tilde{X}_2^4,\,\tilde{X}_2^5,\,\tilde{X}_2^6\right)(\tilde{\psi}_s)\,ds - \int_0^{t/\gamma}\gamma Y_1(\phi_s)\,ds \\[3pt]
&= v'' - v' + \int_0^{t/\gamma}\frac{\cos{\tilde{\psi}_s^1}}{\cos{x_1''}}\gamma Y_1\left((\tilde{\psi}_{t/\gamma,v}^4,\,\tilde{\psi}_{t/\gamma,v}^5,\,\tilde{\psi}_{t/\gamma,v}^6)\right)\,ds - \int_0^{t/\gamma}\gamma Y_1(\phi_s)\,ds \\[3pt]
&= v'' - v' + \int_0^{t/\gamma} \gamma \left[ Y_1\left((\tilde{\psi}_{t/\gamma,v}^4,\,\tilde{\psi}_{t/\gamma,v}^5,\,\tilde{\psi}_{t/\gamma,v}^6)\right) - Y_1(\phi_s)\right]\, ds,
\end{align*}
where the last equality is due to the $x_1$-component of $\tilde{X}_2$ is zero, and so $\tilde{\psi}_s^1 = x_1''$. Again, by Gr\"{o}nwall inequality
\begin{equation}\label{irreducibilitythm3eq4}
\left\vert\left(\tilde{\psi}_{t/\gamma,v}^4,\,\tilde{\psi}_{t/\gamma,v}^5,\,\tilde{\psi}_{t/\gamma,v}^6\right) - \phi_{t/\gamma} \right\vert \leq \vert v' - v''\vert e^{Lt}.
\end{equation}
Combining (\ref{irreducibilitythm3eq3}) and (\ref{irreducibilitythm3eq4}), we have 
\begin{align*}
\left\vert \left(\psi_{t/\gamma}^{\beta,4},\,\psi_{t/\gamma}^{\beta, 5},\,\psi_{t/\gamma}^{\beta, 6}\right) - \phi_{t/\gamma}\right\vert &\leq \left\vert \psi_{t/\gamma}^{\beta} - \tilde{\psi}_{t/\gamma}\right\vert + \left\vert\left(\tilde{\psi}_{t/\gamma,v}^4,\,\tilde{\psi}_{t/\gamma,v}^5,\,\tilde{\psi}_{t/\gamma,v}^6\right) - \phi_{t/\gamma} \right\vert \\[3pt]
&\leq \frac{Mt}{\gamma} e^{Lt/\vert\cos{x}_1''\vert} + \vert v' - v''\vert e^{Lt} \\[3pt]
&\leq \left( \frac{Mt}{\gamma}  + \vert v' - v''\vert\right) e^{Lt/\vert\cos{x}_1''\vert} \\[3pt]
&\leq \left( \frac{Mt}{\gamma}  + \vert v' - v''\vert\right) e^{2Lt}.
\end{align*}
\end{proof}

As a corollary, we have the integral curves of $\mathcal{G}$ can be approximated arbitrarily well by the projection of integral curves of $\mathcal{F}$ on to $\mathbb{S}^2$. 

\begin{corollary}\label{irreducibilitycor1}
Let $\varphi_{\cdot}: [0,\,T]\rightarrow \mathbb{S}^2$ be an integral curve of $\mathcal{G}$ with initial $v'$, and $U$ be an open neighborhood of $\varphi_T$ in $\mathbb{S}^2$. Let $x'\in\mathbb{T}^3$ and $w'= (x',v')$. Then for all $\delta>0$, there exists $0 < T'\leq  T+\delta$, and an integral curve of $\mathcal{F}$, $\psi:[0,\, T']\rightarrow \mathbb{ST}^3$, such that
\begin{itemize}
\item $\psi_0 = w'$,
\item $\left( \psi_{T'}^4,\, \psi_{T'}^5,\, \psi_{T'}^6\right) \in U$.
\end{itemize}
\end{corollary}

\begin{proof}
Let $0= t_0 < t_1 < \dotsc < t_n = T$ be the partition such that for $m=1,\dotsc, n$
\begin{equation*}
\dot{\varphi}_t = Y_{k_m}(\varphi_t)\quad\mbox{for}\quad t_{m-1} < t < t_m,
\end{equation*}
and some $Y_{k_m}\in\mathcal{G}$. Given $\epsilon>0$, we will prove by induction on $n$ that there exist $0< T' \leq T+\delta$ and an integral curve $\psi$ of $\mathcal{F}$ such that
\begin{equation}\label{irreducibilitycor1eq1}
\left\vert\left( \psi_{T'}^4,\, \psi_{T'}^5,\, \psi_{T'}^6\right) - \varphi_{t_n} \right\vert < \epsilon.
\end{equation}

When $n=1$, there are two cases. For the first case,
\begin{equation*}
\varphi_{t_1} = v' + \int_0^{t_1}Y_{k_1}(\varphi_s)\,ds, \qquad Y_{k_1}\in\left\{ Y_{0,1},\,Y_{0,2},\,Y_{0,3},\, Y_{0,4}\right\},
\end{equation*}
then we can apply Theorem~\ref{irreducibilitythm3} and the fact that $\left(\psi_0^4,\,\psi_0^5,\,\psi_0^6\right) = v'$ to get a control $\alpha$ such that
\begin{equation*}
\left\vert \left(\psi_{t_1+\delta}^{\alpha, 4},\,\psi_{t_1+\delta}^{\alpha, 5},\,\psi_{t_1+\delta}^{\alpha, 6}\right) - \varphi_{t_1} \right\vert \leq  \delta M e^{Lt_1},
\end{equation*}
where $\delta>0$ can be arbitrarily small, and the claim is satisfied by taking $\delta= \epsilon M^{-1}e^{-Lt_1}$. For the second case
\begin{equation*}
\varphi_{t_1} = v' + \int_0^{t_1}cY_{k_1}(\varphi_s)\,ds, \qquad c\in\mathbb{R} \quad\mbox{and}\quad Y_{k_1}\in\left\{ Y_1,\,Y_2\right\},
\end{equation*}
and for $\gamma \gg 1$ we can consider the accelerated version of the integral curve,
\begin{equation*}
\phi_{t_1/\gamma} = v' + \int_0^{t_1/\gamma} \gamma c Y_{k_1}(\phi_{s})\,ds. 
\end{equation*}
Then we can apply Theorem~\ref{irreducibilitythm3} and the fact that $\left(\psi_0^4,\,\psi_0^5,\,\psi_0^6\right) = v'$ to get a control $\beta$ such that
\begin{equation*}
\left\vert \left(\psi_{t_1/\gamma}^{\beta,4},\,\psi_{t_1/\gamma}^{\beta, 5},\,\psi_{t_1/\gamma}^{\beta, 6}\right) - \phi_{t_1/\gamma}\right\vert 
\leq \frac{Mt_1}{\gamma}  e^{2Lt_1},
\end{equation*}
where $\gamma>0$ can be arbitrarily large, and the claim is satisfied by taking $\gamma = \epsilon^{-1}Mt_1 \exp{(2Lt_1)}$. Therefore, the base case is proved. 

Now, assume the hypothesis is true for $n=l$, so there exists $T''>0$ such that
\begin{equation*}
\left\vert\left( \psi_{T''}^4,\, \psi_{T''}^5,\, \psi_{T''}^6\right) - \varphi_{t_l} \right\vert < \epsilon;
\end{equation*} 
we will prove that it is also true for $n=l+1$. By taking $\left( \psi_{T''}^4,\, \psi_{T''}^5,\, \psi_{T''}^6\right)$ and $\varphi_{t_l}$ as our new initial conditions, we apply Theorem~\ref{irreducibilitythm3} to extend $\psi$ to get either 
\begin{equation*}
\left\vert \left(\psi_{T''+\delta}^{4},\,\psi_{T''+\delta}^{5},\,\psi_{T''+\delta}^{6}\right) - \varphi_{t_{l+1}} \right\vert \leq \left(\left\vert \left( \psi_{T''}^4,\, \psi_{T''}^5,\, \psi_{T''}^6\right) - \varphi_{t_l} \right\vert + \delta M \right) e^{L(t_{l+1}-t_l)}
\end{equation*}
for $\delta>0$ arbitrarily small, or
\begin{multline*}
\left\vert \left(\psi_{T''+(t_{l+1}-t_l)/\gamma}^{4},\,\psi_{T''+(t_{l+1}-t_l)/\gamma}^{5},\,\psi_{T''+(t_{l+1}-t_l)/\gamma}^{6}\right) - \phi_{t_{l+1}}\right\vert 
\\[3pt]
\leq \left( \left\vert \left( \psi_{T''}^4,\, \psi_{T''}^5,\, \psi_{T''}^6\right) - \varphi_{t_l} \right\vert + \frac{M(t_{l+1} - t_l)}{\gamma}\right) e^{2L(t_{l+1}-t_l)}
\end{multline*}
for $\gamma>0$ arbitrarily large. This concludes our proof. 
\end{proof}

As a consequence, we have the controllability for the projective process. Let us see how it is done before we move on to prove theorem~\ref{irreducibilitythm1}. 

\begin{theorem}
There exists $T'>0$ such that
\begin{equation*}
\overline{\mathcal{A}_{\mathcal{F}}(w, T')} = \mathbb{ST}^3. 
\end{equation*}
\end{theorem}

\begin{proof}
Suppose we want to obtain a control to go from $w' = (x',\, v')$ to a neighborhood $U$ of $w''= (x'',\, v'')$; we may assume that $U$ is of the form $U = \left(I_1\times\dotsc\times I_6\right) \,\cap\, \mathbb{T}^3\times\mathbb{S}^2$, where $I_k$'s are open intervals for $k=1,\dotsc, 6$. By Theorem~\ref{irreducibilitythm2}, there exists an integral curve of $\mathcal{G}$ with the partition $0=t_0<t_1<\dotsc<t_m=T_1$ such that
\begin{equation*}
\phi_0 = v', \quad \phi_{T_1}\in \left(I_4\times I_5\times I_6\right)\cap \mathbb{S}^2 \quad\mbox{and}\quad \phi_{T_1-\epsilon_1} \in\{v_3=0\},
\end{equation*}
where $\epsilon_1>0$ can be arbitrarily small. By definition, 
\begin{equation*}
\dot{\phi}_t = Y_i\in\mathcal{G} \quad\mbox{for}\quad t_{i-1}<t<t_i, \quad i=1,\dotsc, m,
\end{equation*}
and we can apply Corollary~\ref{irreducibilitycor1} to get an integral curve $\varphi= (\varphi^1,\,\varphi^2,\dotsc, \varphi^6): [0,\,T_2]\rightarrow \mathbb{ST}^3$ of $\tilde{\mathcal{F}}$ such that
\begin{itemize}
\item $\left(\varphi_0^4,\,\varphi_0^5,\,\varphi_0^6\right)= v'$. 
\item $\left(\varphi_{T_2}^4,\,\varphi_{T_2}^5,\,\varphi_{T_2}^6\right)\in \left(I_4\times I_5\times I_6\right)\cap \mathbb{S}^2$. 
\item $\left(\varphi_{T_2-\epsilon_2}^4,\,\varphi_{T_2-\epsilon_2}^5,\,\varphi_{T_2-\epsilon_2}^6\right)\in\{v_3=0\}$, where $\epsilon_2$ is a small positive number depending on $\epsilon_1$.
\end{itemize}
It is possible to achieve the third bullet point, since we can pick $\phi$ so that $\phi_u\in\{v_3 > 0\}$ and $\phi_{u'}\in\{v_3 < 0\}$ for some $u < T_1-\epsilon_1 < u'$. By approximating $\phi$ well enough, $\varphi$ will have to cross $\{v_3= 0\}$ by continuity.

Now, observe that both 
\begin{equation*}
\left(\tilde{X}_0^4,\, \tilde{X}_0^5,\, \tilde{X}_0^6\right) =  \left(\tilde{X}_1^4,\, \tilde{X}_1^5,\, \tilde{X}_1^6\right) = 0
\end{equation*}
for $\left\{(x, v)\in\mathbb{T}^3\times\mathbb{S}^2: v_3=0\right\}$; hence, we can use $\tilde{X}_0$ and $\tilde{X}_1$ to control the position in the $x$-direction without changing the position in $v$-direction, once we are in $\left\{(x, v)\in\mathbb{T}^3\times\mathbb{S}^2: v_3=0\right\}$. By Theorem~\ref{irreducibilitythm1}, we construct an integral curve $\tilde{\varphi}:[0,\,T_3+ T_2 - \epsilon_2]\rightarrow \mathbb{ST}^3$ such that
\begin{itemize}
\item $\tilde{\varphi}_t = \varphi_t$ for $0\leq t\leq T_2-\epsilon_2$.
\item $\left(\tilde{\varphi}_t^4,\, \tilde{\varphi}_t^5,\, \tilde{\varphi}_t^6\right) = \left(\tilde{\varphi}_{T_2-\epsilon_2}^4,\,\tilde{\varphi}_{T_2-\epsilon_2}^5,\,\tilde{\varphi}_{T_2-\epsilon_2}^6\right)$  for $T_2-\epsilon_2\leq t\leq T_3+T_2-\epsilon_2$.
\item $\left(\tilde{\varphi}_{T_3+T_2-\epsilon_2}^1,\,\tilde{\varphi}_{T_3+T_2-\epsilon_2}^2,\,\tilde{\varphi}_{T_3+T_2-\epsilon_2}^3\right) \in I_1\times I_2\times I_3$.
\end{itemize}
To finish our proof, first observe that there exists $0 < \epsilon_3 < \epsilon_2$ and open neighborhood $\tilde{U}$ of $\tilde{\varphi}_{T_1+T_2-\epsilon_2}$ such that 
\begin{equation*}
\mathcal{A}_{\tilde{X}_0}(\tilde{U}, \leq \epsilon_3) \subset U.    
\end{equation*}
Again, we may assume $\tilde{U}$ is of the form $\tilde{U} =(\tilde{I}_1\times\dotsc\times \tilde{I}_6) \,\cap\, \mathbb{T}^3\times\mathbb{S}^2$, where $\tilde{I}_k$'s are open intervals for $k=1,\dotsc, 6$. We extend our integral curve to   $\tilde{\varphi}:[0,\,T_3+T_2]\rightarrow \mathbb{ST}^3$ in the following steps:
\begin{enumerate}
\item If necessary, use Theorem~\ref{irreducibilitythm1} to obtain a control such that \begin{equation*}
\left(\tilde{\varphi}_{T_3+T_2-\epsilon_3}^1,\,\tilde{\varphi}_{T_3+T_2-\epsilon_3}^2\right)\in \tilde{I}_1\times\tilde{I}_2
\quad\mbox{and}\quad
\tilde{\varphi}_{T_3+T_2-\epsilon_2}^k = \tilde{\varphi}_{T_3+T_2-\epsilon_3}^k \quad\mbox{for}\quad k = 4,5,6.
\end{equation*}
\item By Theorem~\ref{irreducibilitythm2} and  Theorem~\ref{irreducibilitythm3}, there exists some $k\in\{2,3,4,5\}$ and some $\gamma\gg 1$ such that  
\begin{equation*}
\tilde{\varphi}_{T_3+T_2-\epsilon_3+\gamma^{-1}} = \tilde{\varphi}_{T_3+T_2-\epsilon_3} + \int_{T_3+T_2-\epsilon_3}^{T_3+T_2-\epsilon_3+1/\gamma} \tilde{X}_0(\tilde{\varphi}_s) \pm \gamma\tilde{X}_k(\tilde{\varphi}_s)\,ds,
\end{equation*}
$T_3+T_2-\epsilon_3 + \gamma^{-1} < T_3+T_2$, and 
\begin{equation*}
(\tilde{\varphi}_{T_3+T_2-\epsilon_3+\gamma^{-1}}^4,\,\tilde{\varphi}_{T_3+T_2-\epsilon_3+\gamma^{-1}}^5,\,\tilde{\varphi}_{T_3+T_2-\epsilon_3+\gamma^{-1}}^6) \in \tilde{I}_4\times \tilde{I}_5\times \tilde{I}_6;
\end{equation*}
the sign in front of $\tilde{X}_k$ depends on the direction we are moving.

\item Pick some $\tilde{\gamma}\gg 1$
such that
\begin{equation*}
\tilde{\varphi}_{T_3+T_2-\epsilon_3+\gamma^{-1} +  \tilde{\gamma}^{-1}} = \tilde{\varphi}_{T_3+T_2-\epsilon_3 + \gamma^{-1}} + \int_{T_3+T_2-\epsilon_3+1/\gamma}^{T_3+T_2-\epsilon_3+1/\gamma+ 1/\tilde{\gamma}} \tilde{X}_0(\tilde{\varphi}_s) + \tilde{\gamma}\tilde{X}_1(\tilde{\varphi}_s)\,ds,
\end{equation*}
$T_3+T_2-\epsilon_3 + \gamma^{-1} + \tilde{\gamma}^{-1} < T_3+T_2$ and $\tilde{\varphi}_{T_3+T_2-\epsilon_3+\gamma^{-1}+\tilde{\gamma}^{-1}} \in \tilde{U}$.

\item Finally, consider
\begin{equation*}
\tilde{\varphi}_{T_3+T_2} = \tilde{\varphi}_{T_3+T_2-\epsilon_3 + \gamma^{-1} + \tilde{\gamma}^{-1}} + \int_{T_3+T_2-\epsilon_3+1/\gamma + 1/\tilde{\gamma}}^{T_3+T_2} \tilde{X}_0(\tilde{\varphi}_s)\,ds;
\end{equation*}
due to the way we pick $\epsilon_3$, we have $\tilde{\varphi}_{T_3+T_2} \in U$.
\end{enumerate} 
\end{proof}

We break the proof of Theorem~\ref{irreducibilitythm2} into several lemmas. First, we observe that
\begin{equation*}
    N_0 = \{v_3=0\}, \quad
    N_1 = \{ v_1=0\} \quad\mbox{and}\quad
    N_2 = \{ v_2=0\}
\end{equation*}
are the sets where the vector fields $Y_{0,\cdot}$, $Y_1$ and $Y_2$ vanish respectively. 

\begin{lemma}\label{irreducibilitylemma1}
The non-trivial orbits of $Y_1$ and $Y_2$ cover $\mathbb{S}^2\setminus\{ v_3=\pm 1\}$, 
\begin{equation*}
\mathbb{S}^2\setminus\{v_3= \pm 1\} = \bigcup_{k=1}^2\bigcup_{ \left\{(v,t)\in \mathbb{S}^2\times \mathbb{R}: Y_k(v) \neq 0\right\} } e^{tY_k}(v), 
\end{equation*}
and every non-trivial orbit of $Y_1$ and $Y_2$ intersects $N_0$,
\begin{equation*}
\bigcup_{ t\in\mathbb{R} }  e^{tY_k}(v) \cap N_0 \neq\emptyset, \quad
v\in \left\{v\in\mathbb{S}^2: Y_k(v)\neq 0\right\}, \,k=1,2.
\end{equation*}
\end{lemma}

\begin{proof}
By symmetry, we may focus on studying the orbits of $Y_1$. The orbits of $Y_1$ are partitioned into three parts, 
\begin{equation*}
N_{1,-} = \{ v_1<0\}, \quad  N_1 = \{ v_1=0\} \quad\mbox{and}\quad N_{1,+} = \{ v_1>0\}.
\end{equation*}
The orbits in $ N_1$ are trivial, and the orbits in $N_{1,-}$ and $N_{1,+}$ are symmetric; therefore, we can focus on analyzing $  N_{1,+}$. In particular, we study the projected orbits to the $( v_1, v_3)$-plane,
\begin{align}
    \frac{d v_1}{dt} &= - v_3( v_1)^2 \\[3pt]
    \frac{d v_3}{dt} &=  v_1\left(1-( v_3)^2\right).
\end{align}
We include a picture of the vector field in Figure~\ref{irreducibilitypic2} for reference. 

\begin{figure}
  \centering
  \includegraphics[width=.8\linewidth]{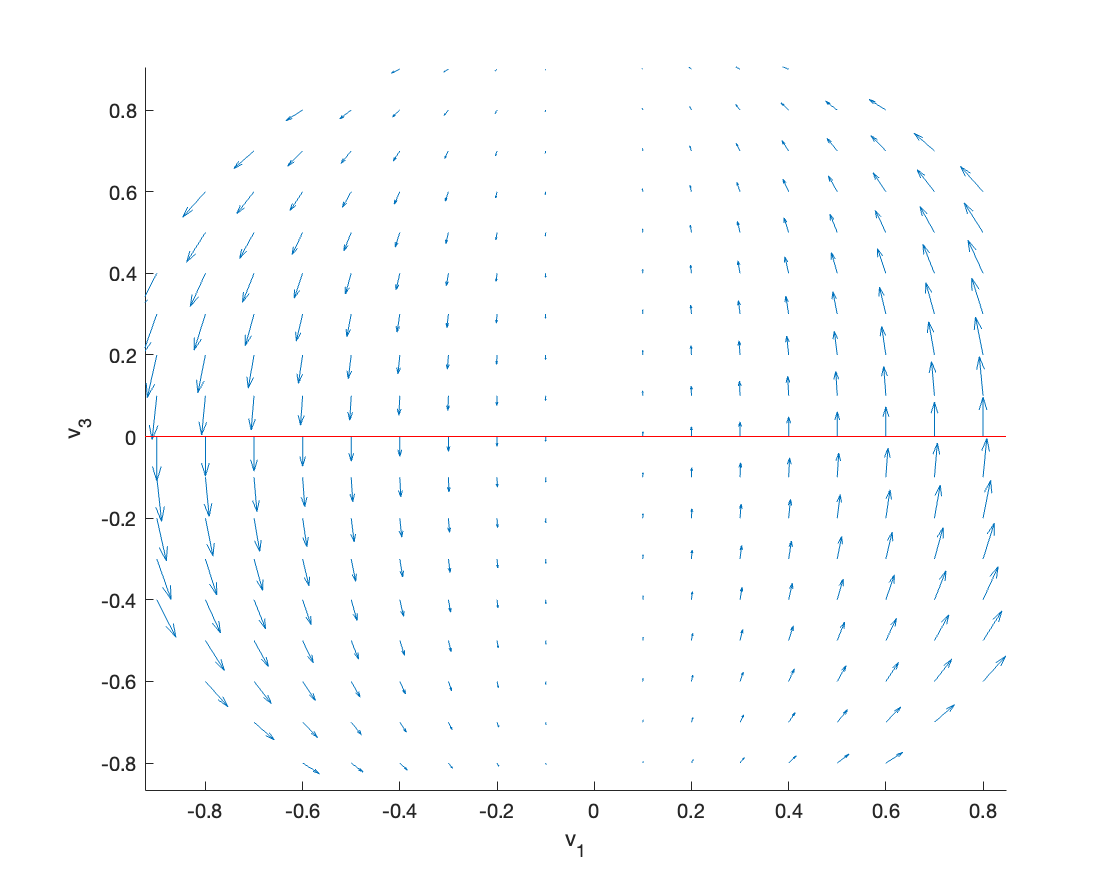}
  \caption{A picture for the vector field $(-v_3(v_1)^2,\, v_1(1-(v_3))^2)$.}
  \label{irreducibilitypic2}
\end{figure}

To show that every orbit in $ N_1,+$ must pass through the equator $N_0$. Indeed,
\begin{align*}
    \frac{d^2 v_3}{dt^2} &= - v_3( v_1)^2(1-( v_3)^2) - 2( v_1)^2 v_3(1-( v_3)^2) \\[3pt]
    &= -3( v_1)^2 v_3(1-( v_3)^2),
\end{align*}
which is positive when $ v_3<0$, zero when $ v_3=0$ and negative when $ v_3>0$; this suggests the system is accelerating upward in the lower half, and accelerating downward in the upper half. This and the fact that
\begin{equation*}
\frac{d v_3}{dt} =  v_1(1-( v_3)^2) > 0 \quad\mbox{in}\quad N_{1,+} 
\end{equation*}
imply every orbit in $N_{1,+}$ passes through $N_0$. 
\end{proof}

Another thing we need is an estimate of how far right one can be when we reach $N_0$, if we start the system from, say, $(v_1, -\epsilon)\in   N_{1,+}$. Let $t$ be the time at which the system hits $N_0$. Then
\begin{align*}
    0 =  v_3(t) &= -\epsilon + \int_0^t  v_1(s)(1-( v_3(s))^2)\,ds \\[3pt]
    &\leq -\epsilon + \int_0^t  v_1(0)(1-\epsilon^2)\,ds = -\epsilon +  v_1(0)(1-\epsilon^2)t,
\end{align*}
where we used the fact that the system is accelerating upward in the lower part of $  N_{1,+}$. This gives an upper bound for $t$,
\begin{equation*}
t \leq \frac{\epsilon}{ v_1(0)(1-\epsilon^2)}.
\end{equation*}
For $ v_1(t)$, we have
\begin{align*}
     v_1(t) &=  v_1(0) + \int_0^t - v_3(s)( v_1(s))^2\,ds \\[3pt]
    &\leq  v_1(0) + \int_0^t \epsilon( v_1(s))^2\,ds \\[3pt]
    &\leq  v_1(0) + \epsilon t \leq  v_1(0) + \frac{\epsilon^2}{ v_1(0)(1-\epsilon^2)}.
\end{align*}
In particular, if we take $( v_1(0), v_3(0)) = (\epsilon,-\epsilon)$, then
\begin{equation}\label{hit}
     v_1(t) \leq \epsilon + \frac{\epsilon^2}{\epsilon(1-\epsilon^2)} < 3\epsilon,
\end{equation}
where the last inequality is true if we take for example $\epsilon^2<1/2$. Estimate (\ref{hit}) gives the following lemma.

\begin{lemma}\label{irreducibilitylemma2}
Let $\mathcal{G}'=\mbox{span}\{Y_1, Y_2\}$. For all $t>0$ and $\epsilon < 1/9$,
\begin{equation*}
N_0 \subset \mathcal{A}_{\mathcal{G}'}(\{ v_3=\pm\epsilon\},\,\leq t).
\end{equation*}
\end{lemma}

Now we are ready to prove theorem~\ref{irreducibilitythm1}. The main idea is to use $Y_{0,\cdot}$ to go across orbits of $  Y_1$ and $  Y_2$ in a subset of $S^2$, then use $Y_1$ and $Y_2$ to move instantaneously within their orbits. 

\begin{proof}
Let $S_{\epsilon, \pm}$ be the stripes where $\epsilon$ will be determined later,
\begin{equation*}
S_{\epsilon, +} = \mathbb{S}^2\cap \{\epsilon \leq  v_3 \leq 1-\epsilon\} \quad\mbox{and}\quad S_{\epsilon, -} = \mathbb{S}^2\cap \{-1+\epsilon \leq  v_3 \leq -\epsilon\}.
\end{equation*}

First, we show that there exists $t_1>0$ such that
\begin{equation}\label{irreducibilityeq1}
S_{\epsilon,+} \subseteq \mathcal{A}_{\mathcal{G}}\left(v, \,\leq t_1\right) \quad\forall v\in S_{\epsilon,+} \quad\mbox{and}\quad S_{\epsilon,-} \subseteq \mathcal{A}_{\mathcal{G}}\left(v, \,\leq t_1\right) \quad\forall v\in S_{\epsilon,-}.
\end{equation}
By symmetry, we focus on studying $S_{\epsilon,+}$. A few properties of the collection $\mathcal{G}$:
\begin{itemize}
\item $Y_{0,1} = -Y_{0,3}$ and $Y_{0,2} = -Y_{0,4}$
\item For all $v\in S_{\epsilon, +}$, span$\{Y_{0,k},\, Y_l\} = T_v\mathbb{S}^2$ for some $(k,l)\in \{(1,1), (1,2), (2,1), (2,2)\}$.
\end{itemize}
Given $v\in S_{\epsilon,+}$, suppose span$\{Y_{0,k},\, Y_l\}=T_v\mathbb{S}^2$. For some neighborhood $U$ of $0$ in $\mathbb{R}^2$, we may consider $g: U\rightarrow S_{\epsilon,+}$, defined as
\begin{equation*}
g(t_1,t_2) = e^{t_2Y_l}\circ e^{t_1Y_{0,k}}(v), 
\end{equation*}
where the exponential denote the flow generated by the vector field. One can see that $g(0,0)=v$, and by letting
\begin{align*}
\varphi_t(v) &= v + \int_0^t Y_{0,k}(\varphi_s(v))\,ds \\[3pt]
\psi_t(v') &= v' + \int_0^t Y_l(\psi_s(v'))\,ds, 
\end{align*}
we have
\begin{equation*}
g(t_1,t_2) = \varphi_{t_1}(v) + \int_0^{t_2}Y_l(\psi_s(\varphi_{t_1}(v)))\,ds.
\end{equation*}
In particular, 
\begin{align*}
\partial_{t_1}g(0,0) &= \left. \dot{\varphi_{t_1}}(v) + \int_0^{t_2} \nabla \left( Y_l\circ\varphi_s\right) \cdot\dot{\varphi_{t_1}}(v)\,ds \right\vert_{(t_1,t_2)=(0,0)} = Y_{0,k}(v) \\
\partial_{t_2}g(0,0) &= Y_l(v),
\end{align*}
and the two vectors are linearly independent; hence, by inverse function theorem, there exists a neighborhood $\tilde{U}$ of $0$ such that $\tilde{U}\subseteq U$ and $g: \tilde{U}\rightarrow g(\tilde{U})$ is a diffeomorphism. Define 
\begin{equation*}
t(v) = 2 \max\{\vert t_1\vert, \vert t_2\vert\}_{(t_1,t_2)\in \tilde{U}}.
\end{equation*}
Since $Y_{0,k}$ and $-Y_{0,k}$ are both in $\mathcal{G}$, the above discussion implies that 
\begin{equation*}
g(\tilde{U}) \subseteq \mathcal{A}_{\mathcal{G}}\left(v, \leq t(v)\right).
\end{equation*}
For each $v\in S_{\epsilon,+}$, there exists a neighborhood of $0$, $U_v\subseteq\mathbb{R}^2$, and a diffeomorphism $g_v: U_v\rightarrow g_v(U_v)$ such that $v=g(0,0)$. Since $S_{\epsilon, +}$ is compact, we may cover it with finitely many open neighborhood $\{g_{v_i}(U_{v_i})\}_{i=1}^N$, and this shows
\begin{equation*}
S_{\epsilon,+} \subseteq \mathcal{A}_{\mathcal{G}}\left(v, \leq \sum_{i=1}^Nt(v_i)\right) \quad\forall v\in S_{\epsilon,+}.
\end{equation*}

Denote 
\begin{equation*}
S_+ = \{v\in \mathbb{S}^2: v_3\geq 0\} \quad\mbox{and}\quad S_- = \{v\in \mathbb{S}^2: v_3\leq 0\}.
\end{equation*}
Next, we show that by picking $\epsilon\ll 1$ there exists $t_2>0$ such that
\begin{equation}\label{irreducibilityeq2}
S_{\epsilon,+} \cap \mathcal{A}_{\mathcal{G}}\left(v, \,\leq t_2\right)\neq\emptyset \quad\forall v\in S_+  \quad\mbox{and}\quad S_{\epsilon,-} \cap \mathcal{A}_{\mathcal{G}}\left(v, \,\leq t_2\right)\neq\emptyset \quad\forall v\in S_-.
\end{equation}
By symmetry, we focus on studying $S_+$. We observe that if $1-\epsilon < v_3 < 1$, then by Lemma~\ref{irreducibilitylemma1}, for any $\delta>0$
\begin{equation*}
N_0 \cap \mathcal{A}_{\mathcal{G}}\left(v, \,\leq \delta\right) \neq \emptyset \quad\forall v\in\{v\in S_+: 1-\epsilon < v_3 < 1\};
\end{equation*}
moreover, since the integral curves are continuous, it also implies
\begin{equation*}
S_{\epsilon,+} \cap \mathcal{A}_{\mathcal{G}}\left(v, \,\leq \delta\right) \neq \emptyset \quad\forall v\in\{v\in S_+: 1-\epsilon < v_3 < 1\}.
\end{equation*}
If $0\leq v_3 < \epsilon$, we may without loss of generality assume $v_3=0$ due to Lemma~\ref{irreducibilitylemma1}. Then by picking $\epsilon<1/9$, we may apply Lemma~\ref{irreducibilitylemma2} to conclude for all $\delta>0$
\begin{equation*}
S_{\epsilon,+} \cap \mathcal{A}_{\mathcal{G}}\left(v, \,\leq \delta\right) \neq \emptyset \quad\forall v\in\{v\in S_+: 0 \leq v_3 < \epsilon\}.
\end{equation*}
If $v_3 = 1$, since $Y_{0,1}\neq 0$ at $(0,0,1)\in S_+$, there exists some $t_2'>0$ such that
\begin{equation*}
v' = e^{t_2'Y_{0,1}}(0,0,1) \neq (0,0,1),
\end{equation*}
and the previous two cases apply to $v'$. In conclusion, we may take $t_2 = t_2'+\delta$ for arbitrary $\delta>0$, and $\epsilon < 1/9$ so that  
\begin{equation*}
S_{\epsilon,+} \cap \mathcal{A}_{\mathcal{G}}\left(v, \,\leq t_2\right)\neq\emptyset \quad\forall v\in S_+.
\end{equation*}

Finally, by combining (\ref{irreducibilityeq1}) and (\ref{irreducibilityeq2}), we obtain
\begin{equation}\label{irreducibilityeq3}
S_{\epsilon,+} \subseteq \mathcal{A}_{\mathcal{G}}\left(v, \,\leq t_1+ t_2\right) \quad\forall v\in S_+ \quad\mbox{and}\quad S_{\epsilon,-} \subseteq \mathcal{A}_{\mathcal{G}}\left(v, \,\leq t_1+t_2\right) \quad\forall v\in S_-.
\end{equation}
By Lemma~\ref{irreducibilitylemma2}, (\ref{irreducibilityeq3}) further implies for all $\delta>0$,
\begin{equation}\label{irreducibilityeq4}
N_0 \subseteq \mathcal{A}_{\mathcal{G}}\left(v, \,\leq t_1+ t_2 + \delta\right) \quad\forall v\in \mathbb{S}^2.
\end{equation}
By Lemma~\ref{irreducibilitylemma1}, (\ref{irreducibilityeq4}) implies $\mathcal{A}_{\mathcal{G}}\left(v, \,\leq t_1+ t_2 + \delta\right)$ is dense in $\mathbb{S}^2$ for all $v$ in $\mathbb{S}^2$. 
\end{proof}

\appendix
\section{Optimality for the decay bound in Theorem~\ref{maintheorem}}\label{sec:AppendixA}

In the following examples, we investigate the sharpness of the decay bound obtained in Theorem~\ref{maintheorem}. The first example shows that the linear decay obtained in Theorem~\ref{maintheorem} is optimal in general. The second example shows that there are cases with sublinear decay. We consider  diffusion processes on $M=\mathbb{T}^2$ with Euclidean metric. 

\begin{example}\label{example1}
Let $M=\mathbb{T}^2$. Consider the process
\begin{equation*}
dx_t^{\epsilon} = \sqrt{\epsilon}\sum_{k=1}^4 X_k(x_t^{\epsilon})\circ dW_k,
\end{equation*}
where 
\begin{equation*}
X_1 = \begin{pmatrix}
1 \\ 0
\end{pmatrix}, \quad 
X_2 = \begin{pmatrix}
0 \\ 1
\end{pmatrix}, \quad
X_3 = \begin{pmatrix}
\sin{x_2} \\ 0
\end{pmatrix} \quad\mbox{and}\quad
X_4 = \begin{pmatrix}
0 \\\sin{x_1} 
\end{pmatrix}.
\end{equation*}
A straightforward calculation shows it is associated with the projective process
\begin{equation*}
dw_t^{\epsilon} = \sqrt{\epsilon}\sum_{k=1}^4 \tilde{X}_k(w_t^{\epsilon})\circ dW_k,
\end{equation*}
where
\begin{equation*}
\tilde{X}_1 = \begin{pmatrix}
1 \\ 0 \\ 0 \\ 0
\end{pmatrix}, \quad
\tilde{X}_2 = \begin{pmatrix}
0 \\ 1 \\ 0 \\ 0
\end{pmatrix}, \quad
\tilde{X}_3 = \begin{pmatrix}
\sin{x_2} \\ 0 \\ (1-v_1^2)v_2\cos{x_2} \\ -v_1v_2^2\cos{x_2}
\end{pmatrix} \quad\mbox{and}\quad
\tilde{X}_4 = \begin{pmatrix}
0 \\ \sin{x_1} \\ -v_1^2v_2\cos{x_1} \\ v_1(1-v_2^2)\cos{x_1}
\end{pmatrix}
\end{equation*}
and we have $\mbox{div}\tilde{X}_1 = \mbox{div}\tilde{X}_2 = 0$, $\mbox{div}\tilde{X}_3 = -4v_1v_2\cos{x_2}$ and $\mbox{div}\tilde{X}_4 = -4v_1v_2\cos{x_1}$. One can check that the processes are hypoelliptic and irreducible. 

First, we show that $\lambda_1^{\epsilon}$ decays exactly linearly as $\epsilon\rightarrow 0$, and therefore our bound cannot be improved in general. Indeed, 
\begin{equation*}
0 = \tilde{\mathcal{L}}^* f^{\epsilon} = \epsilon\sum_k \left(\tilde{X}_k^* \right)^2 f^{\epsilon} = \epsilon\tilde{\mathcal{L}}_1^* f^{\epsilon} \quad\Rightarrow\quad
f^{\epsilon} = f^1,
\end{equation*}
where we used the unique ergodicity. This implies
\begin{equation*}
d\lambda_1^{\epsilon} = \frac{\epsilon}{2}\sum_{k=1}^r \int \frac{\vert\tilde{X}_k^*f^{\epsilon}\vert^2}{f^{\epsilon}}\,dw = \frac{\epsilon}{2}\sum_{k=1}^r \int \frac{\vert\tilde{X}_k^*f^{1}\vert^2}{f^{1}}\,dw. 
\end{equation*}

Next, we show that $\lambda_1^{\epsilon} > 0$ for $\epsilon > 0$. Denote $g^{\epsilon} = \log{f^{\epsilon}}$. By contradiction, we assume $\lambda_1^{\epsilon} = 0$ and therefore by (\ref{preliminaryeq1}),
\begin{equation}\label{exampleeq1}
\tilde{X}_k^* f = 0 \quad\mbox{for}\quad k=1,\dotsc, 4 \quad\Leftrightarrow\quad -\tilde{X}_k g = \mbox{div}\tilde{X}_k \quad\mbox{for}\quad k=1,\dotsc, 4,
\end{equation}
where $f$ is the stationary density for $w_t^{\epsilon}$. Using (\ref{exampleeq1}) for $k=1, 2$, we have $g$ is independent of $x_1$ and $x_2$, and therefore
\begin{equation}\label{exampleeq2}
\nabla g = \begin{pmatrix}
0 \\ 0 \\ \nabla_v g
\end{pmatrix} \quad\mbox{and}\quad \partial_{x_k}\nabla g = \nabla\partial_{x_k}g = 0 \quad\mbox{for}\quad k=1,2.
\end{equation}
Now, consider (\ref{exampleeq1}) for $k=3$, we have 
\begin{equation*}
-\begin{pmatrix}
(1-v_1^2)v_2\cos{x_2} & -v_1v_2^2\cos{x_2}
\end{pmatrix}\cdot \nabla_vg = \mbox{div}\tilde{X}_3 = 0 \quad\mbox{for}\quad (x,v) =  \left(0,\,\frac{\pi}{2},\, \frac{\sqrt{2}}{2},\,\frac{\sqrt{2}}{2}\right);
\end{equation*}
this in particular implies $\nabla_v g(0,\pi/2,\sqrt{2}/2,\sqrt{2}/2) = c(v_1, \,v_2)^t$ for some constant $c$. However, this is absurd, since (\ref{exampleeq2}) would imply
\begin{equation*}
\nabla_v g = c\begin{pmatrix}
v_1 \\ v_2
\end{pmatrix} \quad\mbox{for}\quad (x, v) = \left(0,\,x_2,\, \frac{\sqrt{2}}{2},\,\frac{\sqrt{2}}{2}\right),
\end{equation*}
where $x_2$ can be arbitrary. We may consider for instance $x_2=0$ to see (\ref{exampleeq1}) can not hold, since
\begin{equation*}
\operatorname{div}\tilde{X}_3 = -2,
\end{equation*}
which is not equal to 
\begin{equation*}
-\tilde{X}_3 g = - \frac{\sqrt{2}}{4} \begin{pmatrix}
1 & -1
\end{pmatrix}\cdot \frac{\sqrt{2}}{2} c \begin{pmatrix}
1 \\[3pt] 1
\end{pmatrix} = 0.
\end{equation*}
This is a contradction.
\end{example}

\begin{example}\label{example2}
This example is meant to explain why it is impossible to obtain an upper bound for
\begin{equation*}
\liminf_{\epsilon\rightarrow 0}\frac{\lambda_1^{\epsilon}}{\epsilon}.
\end{equation*}
Let $M=\mathbb{T}^2$. Consider the process
\begin{equation*}
dx_t^{\epsilon} = X_0(x_t^{\epsilon})\,dt + \sqrt{\epsilon}\sum_{k=1}^4 X_k(x_t^{\epsilon})\circ dW_k,
\end{equation*}
where $X_k$ are the same as in Example~\ref{example1}, and $X_0=X_3$; the process is again elliptic and irreducible. By the results in \cite{BBPS22}, the shearing in the deterministic dynamics implies 
\begin{equation*}
\liminf_{\epsilon\rightarrow 0}\frac{\lambda_1^{\epsilon}}{\epsilon} = \infty.
\end{equation*}

\end{example}



\phantomsection
\addcontentsline{toc}{section}{References}
\bibliographystyle{abbrv}
\bibliography{bibliography}

\begin{bibdiv}
\begin{biblist}

\bib{APW86}{article}{
      author={Arnold, Ludwig},
      author={Papanicolaou, George},
      author={Wihstutz, Volker},
       title={Asymptotic analysis of the lyapunov exponent and rotation number
  of the random oscillator and applications},
        date={1986},
     journal={Siam Journal on Applied Mathematics},
      volume={46},
      number={3},
       pages={427\ndash 450},
}

\bib{B89}{article}{
      author={Baxendale, Peter~H},
       title={Lyapunov exponents and relative entropy for a stochastic flow of
  diffeomorphisms},
        date={1989},
     journal={Probability Theory and Related Fields},
      volume={81},
       pages={521\ndash 554},
}

\bib{BG01}{article}{
      author={Baxendale, Peter~H},
      author={Goukasian, Levon},
       title={Lyapunov exponents of nilpotent it\^{o} systems with random
  coefficients},
        date={2001},
     journal={Stochastic Processes and their Applications},
      volume={95},
       pages={219\ndash 233},
}

\bib{BG02}{article}{
      author={Baxendale, Peter~H},
      author={Goukasian, Levon},
       title={Lyapunov exponents for small random perturbations of hamiltonian
  systems},
        date={2002},
     journal={Annals of probability},
       pages={101\ndash 134},
}

\bib{BBPS22}{article}{
      author={Bedrossian, Jacob},
      author={Blumenthal, Alex},
      author={Punshon-Smith, Sam},
       title={A regularity method for lower bounds on the lyapunov exponent for
  stochastic differential equations},
        date={2022},
     journal={Invent. math},
      volume={227},
       pages={429\ndash 516},
}

\bib{BPS22}{article}{
      author={Bedrossian, Jacob},
      author={Punshon-Smith, Sam},
       title={Chaos in stochastic 2d galerkin-navier-stokes},
        date={2022},
     journal={Communications in Mathematical Physics},
      volume={405},
      number={4},
}

\bib{C87}{article}{
      author={Carverhill, Andrew},
       title={Furstenberg's theorem for nonlinear stochastic system},
        date={1987},
     journal={Porbability theory and related fields},
      volume={74},
       pages={529\ndash 534},
}

\bib{DNR11}{article}{
      author={Deville, R. E.~Lee},
      author={Namachchivaya, N.~Sri},
      author={Rapti, Zoi},
       title={Stability of a stochastic two-dimensional non-hamiltonian
  system},
        date={2011},
     journal={Siam Journal on Applied Mathematics},
      volume={71},
      number={4},
       pages={1458\ndash 1475},
}

\bib{AM82}{article}{
      author={E.I.Auslender},
      author={Mil'shtein, G.N.},
       title={Asymptotic expansions of the liapunov index for linear stochastic
  systems with small noise},
        date={1982},
     journal={Journal of Applied Mathematics and Mechanics},
      volume={46},
      number={3},
       pages={227\ndash 283},
}

\bib{F63}{article}{
      author={Furstenberg, H.},
       title={Noncommuting random products},
        date={1963},
     journal={Trans. AMS},
      volume={108},
       pages={377\ndash 428},
}

\bib{FK60}{article}{
      author={Furstenberg, H.},
      author={Kesten, H.},
       title={Products of random matrices},
        date={1960},
     journal={The Annals of Mathematical Statistics},
      volume={31},
      number={2},
       pages={457\ndash 469},
}

\bib{HM15}{article}{
      author={Herzog, David~P.},
      author={Mattingly, Jonathan~C.},
       title={A practical criterion for positivity of transition densities},
        date={2015},
     journal={Nonlinearity},
      volume={28},
}

\bib{J97}{book}{
      author={Jurdjevic, V.},
       title={Geometric control theory},
   publisher={Cambridge university press},
        date={1997},
}

\bib{K97}{book}{
      author={Kunita, H.},
       title={Stochastic flows and stochastic differential equations},
   publisher={Cambridge university press},
        date={1997},
}

\bib{L12}{book}{
      author={Lee, J.~M.},
       title={Introduction to smooth manifolds},
   publisher={Springer New York, NY},
        date={2012},
}

\bib{OR73}{book}{
      author={Ole\v{i}nik, O.~A.},
      author={Radkevi\v{c}, E.V.},
       title={Second order equations with nonnegative characteristic form},
   publisher={American Mathematical Society},
        date={1973},
}

\bib{O68}{article}{
      author={Oseledets, Valery~Iustinovich},
       title={A multiplicative ergodic theorem. characteristic ljapunov,
  exponents of dynamical systems},
        date={1968},
     journal={Trudy Moskovskogo Matematicheskogo Obshchestva},
      volume={19},
       pages={179\ndash 210},
}

\bib{PW07}{article}{
      author={Pinsky, M.A.},
      author={Wihstutz, V.},
       title={Lyapunov exponents of nilpotent it\^{o} systems},
        date={1998},
     journal={Stoachstics},
      volume={25},
      number={1},
       pages={43\ndash 57},
}

\bib{DPZ96}{book}{
      author={Prato, G.~Da},
      author={Zabczyk, J.},
       title={Ergodicity for infinite dimensional systems},
   publisher={Cambrdige University Press},
        date={1996},
}

\bib{R79}{article}{
      author={Raghunathan, Madabusi~S},
       title={A proof of oseledec’s multiplicative ergodic theorem},
        date={1979},
     journal={Israel Journal of Mathematics},
      volume={32},
       pages={356\ndash 362},
}

\bib{S00}{article}{
      author={Sowers, Richard~B.},
       title={On the tangent flow of a stochastic differential equation with
  fast drift},
        date={2000},
     journal={Transactions of the American Mathematical Society},
      volume={353},
      number={4},
       pages={1321\ndash 1334},
}

\bib{SV72}{article}{
      author={Stroock, D.W},
      author={Varadhan, S R~S},
       title={On the support of diffusion processes with applications to the
  strong maximum principle},
        date={1972},
     journal={Proceedings of the Sixth Berkley Symposium on Mathematical
  Statistics and Probability},
      volume={3},
       pages={333 \ndash  359},
}

\bib{V84}{book}{
      author={Varadarajan, V.S.},
       title={Lie groups, lie algebras, and their representation},
   publisher={Springer},
        date={1984},
}

\bib{W93}{article}{
      author={Walters, Peter},
       title={A dynamical proof of the multiplicative ergodic theorem},
        date={1993},
     journal={Transactions of the American Mathematical Society},
      volume={335},
      number={1},
       pages={245\ndash 257},
}

\end{biblist}
\end{bibdiv}

\end{document}